\documentclass[11pt,a4paper]{article}
\usepackage[utf8]{inputenc}
\usepackage{amsmath}
\usepackage{amsfonts}
\usepackage{titlesec}
\usepackage[capposition=top]{floatrow}
\usepackage[margin=1.2in]{geometry} 
\usepackage{changepage}
\usepackage{amssymb}
\usepackage{relsize}
\usepackage{xcolor}
\usepackage{marginnote}
\usepackage{mathtools}
\usepackage{amsthm}
\usepackage{bbm}
\usepackage{bm}
\usepackage{shadethm}
\usepackage{expl3}
\usepackage{graphicx}
\usepackage{url} 
\usepackage{setspace}
\usepackage{braket}
\usepackage{epstopdf}
\usepackage{titlesec}
\usepackage{hyperref}
\usepackage{ulem}
\usepackage{extpfeil}

\usepackage{OldStandard} 
\renewcommand{\epsilon}{\varepsilon}

\let\originalleft\left 
\let\originalright\right
\renewcommand{\left}{\mathopen{}\mathclose\bgroup\originalleft}
\renewcommand{\right}{\aftergroup\egroup\originalright}

\newcommand{\prodprime}{\sideset{}{'}\prod}

\newcommand{\dd}{\mathrm{d}}

\newcommand{\lb}{\left[}
\newcommand{\rb}{\right]}
\newcommand{\lp}{\left(}
\newcommand{\rp}{\right)}
\newcommand{\lbr}{\left\lbrace}
\newcommand{\rbr}{\right\rbrace}

\newcommand{\hsp}{\hspace{0.1cm}}

\newcommand{\trace}{\mathrm{Tr}}
\newcommand{\R}{{\bf R}}
\newcommand{\Z}{{\bf Z}}

\newcommand{\Q}{{\bf Q}}
\newcommand{\A}{\mathbb{A}}
\newcommand{\Hh}{\mathbb{H}}
\newcommand{\C}{{\bf C}}
\newcommand{\F}{\mathbb{F}}
\newcommand{\adele}{\mathbb{A}}
\newcommand{\G}{\mathbb{G}}
\newcommand{\SL}{\mathrm{SL}}
\newcommand{\SP}{\mathrm{Sp}}
\newcommand{\GSP}{\mathrm{GSp}}
\newcommand{\GL}{\mathrm{GL}}
\newcommand{\SO}{\mathrm{SO}}
\newcommand{\SU}{\mathrm{SU}}

\newcommand{\yoyo}{{2n}}
\newcommand{\ringint}{\mathcal{O}}

\newcommand{\adjustin}{\begin{adjustwidth}{4em}{4em}}
\newcommand{\adjustout}{\end{adjustwidth}}

\renewcommand{\xRightarrow}[2][]{\ext@arrow 0359\Rightarrowfill@{#1}{#2}}
\numberwithin{equation}{section}

\newtheorem{theorem}{Theorem}
\newcounter{definition}
\newtheorem{prop}[theorem]{Proposition}
\newtheorem{lemma}[theorem]{Lemma}
\newtheorem{corollary}[theorem]{Corollary}

\numberwithin{theorem}{section}
\numberwithin{definition}{section}

\allowdisplaybreaks

\usepackage{listings}
\usepackage{color}

\definecolor{dkgreen}{rgb}{0,0.6,0}
\definecolor{gray}{rgb}{0.5,0.5,0.5}
\definecolor{mauve}{rgb}{0.58,0,0.82}

\begin{document}
\begin{center}
    \textbf{THE SIEGEL MODULAR GROUP IS THE LATTICE OF} \\
    \textbf{MINIMAL COVOLUME IN THE SYMPLECTIC GROUP} \\[+0.5em]
    Amir Džambić, Kristian Holm, Ralf Köhl \\[+0.5em]
    \today
\end{center}
\adjustin
\begin{small}
    \textbf{Abstract.} Let $n \geqslant 2$. We prove that, up to conjugation, $\SP_\yoyo(\Z)$ is the unique lattice in $\SP_\yoyo(\R)$ of the smallest covolume.
\end{small}
\adjustout

\tableofcontents

\section{Introduction}
A lattice in a semisimple Lie group $G$ is a discrete subgroup $\Gamma$ such that the $G$-invariant measure on the quotient $G / \Gamma$ is finite. For example, $\SL_n (\Z)$ is a lattice in $\SL_n (\R)$, and $\SP_\yoyo (\Z)$ is a lattice in $\SP_\yoyo (\R)$, for any $n \geqslant 1$. In fact, by the work of Siegel, the covolumes of these lattices are known exactly: When the measures $\mu$ on $G = \SL_n (\R), \SP_{2n}(\R)$ are suitably normalized, one has the beautiful formulas 
\begin{align} \label{siegelvolumeformulas}
    \mu (G / \Gamma) = \begin{cases}
        \zeta(2) \zeta(3) \cdots \zeta(n), &\text{if } G = \SL_n (\R),\\[+0.3em]
        \zeta(2) \zeta(4) \cdots \zeta(2n), &\text{if } G = \SP_\yoyo (\R),
    \end{cases}
\end{align}
where $\Gamma = G \cap \GL_n(\Z)$ in either case, and $\mu ( G / \Gamma)$ denotes the volume of a fundamental domain of $\Gamma$ in $G$. However, both the special linear and the symplectic group contain many more lattices than these standard examples. \par In this paper, we investigate lattices of \textit{minimal} covolume; that is, lattices which are as dense as possible in their ambient groups. Moreover, we focus our attention on $\SP_\yoyo (\R)$. An obvious question is whether such a lattice necessarily exists --- for example, in the case of the group $\R^n$, one may find lattices of arbitrarily large or small covolume. The setting when one considers a semisimple Lie group such as $G = \SP_\yoyo(\R)$, however, is more rigid, as the Kazdan--Margulis Theorem \cite{kazmar} shows that lattices of minimal covolume always exist. Clearly, such a lattice has to be maximal, i.e. not properly contained in any other lattice. Thus, we let $\Gamma \subseteq \SP_\yoyo(\R)$ be a lattice with minimal covolume. \par 
By the celebrated arithmeticity theorem of Margulis (see e.g. \cite[Thm. 16.3.1]{morris}), any irreducible lattice $\Lambda$ in a semisimple Lie group $G$ of real rank greater than $1$ is also an arithmetic subgroup of $G$. In particular, if such a Lie group $G$ is a simple, any lattice is arithmetic. That is, there exists a $\Q$-embedding $\iota : G \rightarrow \GL_n (\R)$ such that $\Lambda$ is commensurable to the group $G(\Z) = \iota^{-1} (\iota (G) \cap \GL_n (\Z) )$ of integer points of $G$. For our purposes, the utility of this identification of lattices with arithmetic subgroups is that arithmetic subgroups are, in some sense, easier to describe. However, as any individual arithmetic subgroup of $G$ is associated to a corresponding embedding of $G$ into $\GL_n (\R)$, investigating all arithmetic subgroups involves the abstraction of treating all $\Q$-isomorphic copies of $G$ in $\GL_n (\R)$ equally. \par 
We therefore let $\G$ be an algebraic group over a number field $K$, for which there exists an embedding $\upsilon_0 : K \rightarrow \R$ which induces an $\R$-isomorphism from $\SP_\yoyo$ (the standard split form of the symplectic group) to $\G$. This is enough to capture any candidate for $\Gamma$ that we will want to consider. Moreover, it leads to a useful alternative characterization of $\Gamma$: Rohlfs gave a cohomological criterion \cite[Satz 3.5]{rohlfs} (see also \cite[Prop. 1.4]{borelprasad}) for maximality of arithmetic subgroups, according to which the lattice $\Gamma \subseteq \SP_\yoyo (\R)$ must be conjugate to the normalizer of a so-called \textit{principal arithmetic subgroup} $\Lambda$ of $\G (K)$. Such subgroups are distinguished by being topologically well-behaved in relation to the group $\G (\A_K)$ of adelic points of $\G$. In particular, $\Lambda$ is given by the intersection $\G(K) \cap \Pi_{\upsilon < \infty} P_\upsilon$ for a certain (topologically coherent) family of compact-open (more precisely, \textit{parahoric}) subgroups $P_\upsilon \subseteq \G (K_\upsilon)$, where $\upsilon$ ranges over all finite places of $K$. These facts make the problem of computing the covolume of $\Lambda$, and thereby identifying a lattice $\Gamma$ of minimal covolume, amenable to a wide range of analytic and algebraic tools. \par 
The main tool for studying the covolume of principal arithmetic subgroups of certain well-behaved algebraic groups $\G$ is \textit{Prasad's volume formula}, which expresses the covolume of $\Lambda$ (identified with its diagonal embedding) in
\begin{align*}
\G_\infty := \prod_{\upsilon \mid \infty} \G (K_\upsilon)  
\end{align*}
in terms of a variety of arithmetic invariants of the number field $K$ and, additionally, in terms of the product of certain measures of the parahoric subgroups $\lbr P_\upsilon : \upsilon < \infty \rbr$ (see Section \ref{subsectionPFV} for more details). As such, it has been used in several works to identify lattices of minimal covolume in different Lie groups. Many concrete instances of this problem have been studied before: The case $\SL_2 (\C)$ was studied from different perspectives by Meyerhoff \cite{meyerhoff}, Gehring--Martin \cite{gehringmartin}, and Marshall--Martin \cite{marshallmartin}. Lubotzky \cite{lubotzky} considered this group over function fields and studied minimal covolume lattices of $\SL_2 (\mathbb{F}_q (( t^{-1} )) )$. For certain indefinite orthogonal groups, lattices of minimal covolume were identified by Belolipetsky \cite{belolipetsky} (in the case of $\SO(n,1)$) and Belolipetsky--Emery \cite{belolipetskyemery} (for $\mathrm{PO}(n,1)^\circ$). Also, Emery--Kim \cite{emerykim} considered the analogous question for the indefinite symplectic group $\SP(1,n)$. Most recently, Thilmany \cite{thilmany} proved that in the special linear groups $\SL_n (\R)$ with $n \geqslant 3$, the lattices of minimal covolume are precisely the conjugates of the standard lattice $\SL_n (\Z)$. By contrast, for the case $n = 2$ it was already known to Siegel \cite{siegelSL2} that the $(2,3,7)$ triangle group $\langle s, t \mid s^2 = t^3 = (st)^7 = 1 \rangle$ has minimal covolume in $\SL_2 (\R)$. In light of Thilmany's result about $\SL_n (\R)$ and the fact that $\SL_2 (\R) = \SP_2(\R)$, it is natural to ask about the corresponding minimal covolume lattices in $\SP_\yoyo (\R)$. The purpose of this paper is to answer this question by proving the following theorem. 
\begin{theorem}\label{maintheorem1}
    For $n \geqslant 2$, let $\Gamma \subseteq \SP_\yoyo (\R)$ be a lattice of minimal covolume. Then for some $g \in \SP_\yoyo (\R)$, $\Gamma = g^{-1} \SP_\yoyo (\Z) g$.
\end{theorem}
\subsection{Outline of the argument}
We now describe the method of proof, which is based on \cite{thilmany} and other works in which questions of arithmetic lattices of minimal covolume are investigated.\par 
By the symplectic part of the formula \eqref{siegelvolumeformulas}, we know the covolume of $\Gamma_0 := \SP_\yoyo (\Z)$ in $\SP_\yoyo (\R)$ explicitly. The idea is then to use this number as a reference, which is to be compared to the covolume $\mu (\G_\infty / \Lambda )$ assigned by Prasad's volume formula to an arbitrary principal arithmetic subgroup $\Lambda$ of $\G (K)$, and thus, with the help of the relation
\begin{align*}
    \mu (\G_\infty / \Gamma) = \frac{1}{\lb \Gamma : \Lambda \rb} \mu (\G_\infty / \Lambda),
\end{align*}
to the covolume of $\Gamma$. However, a subtle point here is that the parametrization of the Haar measure on the maximal compact subgroup $\SO(2n) \cap \SP_\yoyo (\R)$ implicit in the equality \eqref{siegelvolumeformulas} is, in fact, different from the parametrization implicit in the volume formula; hence our reference covolume will have to be corrected with the appropriate factor, which 
turns out to be $\Pi (n) := \prod_{j = 1}^n (2 \pi)^{-2j}(2j-1)!$. Thus, we will instead be comparing $\mu (\G_\infty / \Gamma)$ to
\begin{align}\label{referencecovolume}
\Psi(n) := \mu (\SP_\yoyo (\R) / \Gamma_0) = \prod_{j = 1}^n \zeta (2j) \frac{(2j-1)!}{(2 \pi)^{2j}} = \Pi (n) \prod_{j = 1}^n \zeta (2j).
\end{align}
This comparison involves both global and local aspects, for which all relevant parameters are introduced and discussed in Section \ref{SectionPreliminaries}. \par 
As far as our global considerations are concerned, Prasad's Volume Formula involves certain arithmetic invariants of the number field $K$, such as the absolute value $D_K$ of the discriminant of $K$ and the degree $d_K := \lb K : \Q \rb$. As we are interested in $\Gamma$ of \textit{minimal} covolume, the condition $\mu (\SP_\yoyo (\R) / \Gamma) \leqslant \Psi(n)$, which results from the minimality we demand of $\mu(\SP_\yoyo (\R)/ \Gamma)$, will eventually put a number of restrictions on these arithmetic invariants, effectively narrowing down the list of possible number fields $K$ that can realize $\Gamma$ (in the sense of the discussion above). \par 
The restrictions mentioned above are described explicitly in Section \ref{SectionBoundingthecovol}, following a series of lower bounds on the covolume $\mu (\G (K_{\upsilon_0}) / \Gamma)$, where $\G(K_{\upsilon_0})$ denotes the image of $\SP_{2n} (\R)$ under the isomorphism induced by the embedding $\upsilon_0$ discussed above. These bounds are the result of a form of compromise: on the one hand, it is difficult to keep track of too many parameters (e.g. the regulator $R_K$, the class number $h_K$, etc.) at the same time, so it is convenient to eliminate as many of these as we can by estimating them in terms of $d_K$; on the other hand, we also want the precision of the lower bounds that \textit{do} involve these invariants. As such, we will initially use our crudest bounds to obtain upper bounds on $d_K$ and $D_K$ depending on the rank $n$ of our group, and then subsequently apply detailed estimates for each single case where the arithmetic invariants are no longer abstract parameters, but concrete numbers that we can either compute or look up. \par 
Once the number of admissible number fields $K$ is sufficiently small, we can take the relevant local aspects into account and analyze them on a case-by-case basis to determine if they can give rise to a lattice of minimal covolume. After identifying $\Q$ as the only possibility, we conclude that $\G$ must be the split form of the symplectic group, and that $\Gamma$ must be a conjugate of $\SP_\yoyo (\Z)$. This is accomplished in Section \ref{sectionlocalcontributions}.
\subsection{Notation and conventions}
We will use the following notation throughout:
\begin{itemize}
\item $\Q$, $\R$, and $\C$ will denote the fields of rational, real, and complex numbers (respectively). For a prime number $p$, $\Q_p$ will denote the field of $p$-adic numbers, and $\Z_p$ its ring of integers. $\F_q$ will denote the finite field with $q$ elements. 
\item $K$ will denote a number field, and $\A_K$ will denote the corresponding ring of adeles. The degree of $K$ (over $\Q$) will be denoted by $d_K$, and the absolute value of its discriminant (over $\Q$) will be denoted by $D_K$. The notation $\overline{K}$ will be used to denote an algebraic closure of $K$.
\item $\G$ will denote a linear algebraic group defined over $K$, which is a real form of $\SP_\yoyo (\R)$. Moreover, $\mathcal{G}$ will denote the unique quasisplit inner $K$-form of $\G$. \par 
Without necessarily mentioning it explicitly, we will occasionally consider $\G(K)$ and its subgroups as embedded diagonally into $\G(\adele_K)$ whenever this is appropriate.
\item The symbol $\upsilon$ will be used to denote a place of $K$, either infinite (in which case we write $\upsilon \mid \infty$) or finite (in which case we write $\upsilon < \infty$). Moreover, $\upsilon_0$ will denote a distinguished real place of $K$, over which our algebraic group $\G$ splits, i.e. $\G (K_{\upsilon_0}) \simeq \SP_\yoyo (\R)$. \par 
\item $P_\upsilon$ (where $\upsilon$ denotes a finite place of $K$) will denote a parahoric subgroup of $\G (K_\upsilon)$.
\item $\Gamma \subseteq \SP_\yoyo (\R)$ will denote a lattice of minimal covolume defined by the $K$-form $\G$ of $\SP_\yoyo$. (We will also consider the gamma function and denote it by $\Gamma$, but this will cause no ambiguity.) $\Gamma_0$ will denote the Siegel modular group $\SP_\yoyo(\Z)$. 
\item $\Lambda$ will denote the principal arithmetic subgroup of $\G_\infty$ defined by the collection of parahorics $\lbr P_\upsilon : \upsilon < \infty \rbr$, for which $\Gamma$ is its normalizer in the group $\G (\R)$ (cf. \cite[Prop. 1.4.iv)]{borelprasad}).
\item $\zeta$ will denote the Riemann zeta function, and $\zeta_K$ will denote the Dedekind zeta function of the number field $K$. 
\end{itemize}
\section{Preliminaries}\label{SectionPreliminaries}
We will now recall a number of definitions and results from algebraic number theory and the theory of algebraic and arithmetic groups. We end this section by describing the main technical tool needed in this paper, namely the volume formula due to G. Prasad.
\subsection{Number fields}\label{subsectionnumberfields}
Let $K$ be an algebraic number field, i.e. an extension of the rational numbers of degree $d_K < \infty$, and let $D_K$ be the absolute value of its discriminant. \par
The discriminant is of course unbounded as a function of the degree, but it is possible to give lower bounds for $D_K$ in terms of $d_K$. For our purposes, the series of bounds provided by Odlyzko for totally real $K$ will suffice: These bounds have the form
\begin{align}\label{odlyzkobounds}
    D_K > A^{d_K} e^{-E}, 
\end{align}
where the pair $(A,E)$ can be chosen from the list \cite{odlyzkowebsite}. \\ \par  
In the remainder of this section, we will assume that $K$ is totally real. This means that $K$ has precisely $r = d_K$ distinct embeddings $\sigma_1, \ldots, \sigma_{r}$ into $\C$ with $\sigma_i (K) \subseteq \R$ for all $i$. If $\ringint_K \subseteq K$ is the ring of algebraic integers of $K$, the multiplicative group $\ringint_K^\times$ of units is a finitely generated abelian group and can therefore be decomposed as $\ringint_K^\times \simeq \mu (K) \times U_K$, where $\mu (K)$ denotes the torsion subgroup, consisting of units of finite order, and $U_K$ denotes the free part. Then the statement of Dirichlet's unit theorem \cite[Thm. I.7.3]{neukirch} is that $U_K$ has rank $r - 1 = d_K - 1$. Let us suppose that $K \neq \Q$ so that this rank is non-zero. Then, any set of $r-1$ units $\lbr \varepsilon_1, \ldots, \varepsilon_{r-1} \rbr \subseteq \ringint_K^\times$ that generate the free group $U_K$ is called a \textit{fundamental system of units} or a collection of \textit{fundamental units}. For example, if $K$ is a quadratic number field, then a fundamental unit is given by 
\begin{align}\label{quadnffundamentalunit}
\varepsilon = \frac{a + b\sqrt{D_K}}{2},
\end{align}
where $(a,b) \in \Z_+^2$ is the smallest pair of integers satisfying Pell's equation $a^2 - D_K b^2 = \pm 4$.
\par Two particular subgroups of $U_K$ will be of special interest to us, namely the group of \textit{totally positive} units
\begin{align*}
U_K^+ := \lbr u \in \ringint_K^\times : \sigma_i (u) > 0 \text{ for } i = 1, \ldots, d_K \rbr,
\end{align*}
and the group $U_K^2$ of squares of units. If $u^2 \in U_K^2$ is any such square, then naturally $u^2$ is totally positive as $\sigma_i (u) \in \R$ for each $i$; that is, $U_K^2 \subseteq U_K^+$ is a subgroup. For its index we will need the following facts, which are straightforward consequences of Dirichlet's unit theorem:
\begin{itemize}
\item With no assumptions on the totally real number field $K$, one always has
\begin{align}\label{totallypositivemodsquares}
\lb U_K^+ : U_K^2 \rb = \left| U_K^+ / U_K^2 \right| \leqslant 2^{d_K + r_K - 1} = 2^{2 d_K - 1}.
\end{align}
\item If $K$ is a real quadratic number field with fundamental unit $\varepsilon$, it is known that if $\varepsilon$ is not totally positive, then 
\begin{align}\label{quadraticfieldnottotallypositiveFU}
\lb U_K^+ : U_K^2 \rb = 1.
\end{align} 
\end{itemize}
\par Under the Minkowski embedding 
\begin{align*}
\log^+ \! \circ \, \bm{\sigma} = (\log^+ \! \circ \, \sigma_1, \ldots, \log^+ \!  \circ \, \sigma_{r}) : K^\times \longrightarrow \R^{r}, \quad \quad \hspace{-1em}\alpha \longmapsto (\log | \sigma_1 (\alpha) |, \ldots, \log | \sigma_r (\alpha) | ),
\end{align*}
the image $\bm{\sigma}(U_K) \subseteq \R^{r}$ of the multiplicative group $U_K$ is a free additive group of rank $r-1$, or in other words, an $(r-1)$-dimensional lattice in its $(r-1)$-dimensional ambient space $\bm{\sigma}(U_K) \otimes \R \subseteq \R^r$. (A collection of fundamental units could then also be defined as the preimage of any basis for $\bm{\sigma}(U_K)$ under $\log^+ \! \circ \, \bm{\sigma}$.) Denoting the (finite) covolume of this lattice in $\bm{\sigma}(U_K) \otimes \R$ by 
\begin{align*}
V_K := \mathrm{vol}\lp ( \bm{\sigma}(U_K) \otimes \R ) / \bm{\sigma}(U_K) \rp,
\end{align*}
we obtain the \textit{regulator} of $K$ as the scaled covolume $R_K = V_K / \sqrt{r} = V_K / \sqrt{d_K}$. (In case $K = \Q$, one defines $R_K := 1$.) It is a useful fact that this quantity can be bounded from below in terms of the degree $d_K$ of $K$. Indeed, Zimmert proved \cite{zimmert} that for any totally real number field $K$,
\begin{align}\label{zimmert}
R_K \geqslant 0.04 e^{0.46 \cdot d_K}.
\end{align}
\par Although the regulator of $K$ does not appear in relation to Prasad's volume formula, it is closely related to an invariant of $K$ that appears in the crucial estimate \eqref{indexboundgaloiscohomology} of the index $\lb \Gamma : \Lambda \rb$ (see Section \ref{subsectionPFV}). This invariant is the \textit{class number} $h_K$ of $K$ (see \cite[§I.6]{neukirch}). The relation between $R_K$ and $h_K$ is given by the following version of the Brauer--Siegel formula, which appears at various places in the literature concerning the covolumes of arithmetic groups (e.g. \cite[eq. (6.1)]{borelprasad}).
\begin{prop}[{Corollary to \cite{siegel35}}]\label{brauersiegel}
Let $K$ be a totally real number field of degree $d_K = \lb K : \Q \rb$, and let $\zeta_K$ denote the Dedekind zeta function of $K$. Let $D_K$ denote the absolute value of the discriminant of $K$, and let $h_K$ and $R_K$ denote the class number and the regulator of $K$, respectively. Then for any $t > 0$, one has the estimate
\begin{align*}
R_K h_k \leqslant t(t+1)2^{1-d_K} \lp \pi^{-d_K} D_K \rp^{(1+t)/2} \Gamma \lp \frac{1 + t}{2} \rp^{d_K} \zeta_K (1+t).
\end{align*}
\end{prop}
\begin{proof}
We initially define the notation
\begin{align*}
N(\textbf{x}) = x_1 x_2 \cdots x_{d_K}, \quad \quad \trace (\textbf{x}) = x_1 + x_2 + \cdots + x_{d_K}, \quad \quad \textbf{x} = (x_1, \ldots, x_{d_K}) \in \R^{d_K}.
\end{align*}
If we let $\textbf{H} = \lbr \textbf{x} \in \R^{d_K} : N(\textbf{x}) \geqslant 1 \rbr$ and $\lambda = \sqrt{D_K} \cdot \mathop{\mathrm{res}}_{s = 1} \zeta_K (s)$, we obtain from \cite[Lemma 1]{siegel35} that, for any $s \in \C$,
\begin{align*}
&\pi^{- {d_K} s/2} D_K^{s/2} \Gamma \lp \frac{s}{2} \rp^{d_K} \zeta_K (s) 
\\ &\quad \quad = \frac{\lambda}{s (s-1)} + \sum_{\mathfrak{m} \subseteq \mathcal{O}_K} \int_\textbf{H} \lp N(\textbf{x})^{s/2} + N(\textbf{x})^{(1-s)/2} \rp e^{-\pi N(\mathfrak{m})^{2/{d_K}} D_K^{-1/{d_K}} \trace(\textbf{x})} \hsp \frac{\dd x_1}{x_1} \cdots \frac{\dd x_d}{x_d},
\end{align*}
since $K$ is assumed to be totally real. \par 
The definition of $\textbf{H}$ implies that the sum over ideals $\mathfrak{m} \subseteq \mathcal{O}_K$ is positive. If we suppose that $s$ is real and satisfies $s > 1$, then we have
\begin{align}\label{ineq1}
\pi^{-{d_K} s/2} D_K^{s/2} \Gamma \lp \frac{s}{2} \rp^{d_K} \zeta_K (s) \geqslant \frac{\lambda}{s (s-1)}, \quad \quad s > 1.
\end{align}
By the class number formula \cite[Corollary 5.11]{neukirch}, we have $\lambda = 2^{{d_K}-1} R_K h_K$. The claim now follows from this by substituting $s = t + 1$ in (\ref{ineq1}), where $t > 0$ is arbitrary.
\end{proof}
In order to handle the gamma function appearing in the statement of Proposition \ref{brauersiegel}, we record the following “pointwise” version of Stirling's formula due to Robbins \cite{robbins}. Thus, if $n \geqslant 1$ is arbitrary, one has the estimates 
\begin{align}\label{pointwisestirling}
   \sqrt{2 \pi n} \lp \frac{n}{e} \rp^n e^{1/(12n + 1)} < n! < \sqrt{2 \pi n} \lp \frac{n}{e} \rp^n e^{1/12n}.
\end{align}
\par We will also need to consider the adele ring $\A_K$ of $K$. To describe this ring, we recall that a \textit{place} of $K$ is an equivalence class of valuations on $K$, where two valuations are deemed equivalent if they induce the same topology on $K$. Suppressing the formality of equivalence classes, we say that a \textit{finite} place $\upsilon$ of $K$ is such a valuation $\left| \, \cdot \, \right|_\upsilon$ which extends a $p$-adic valuation on $\Q$ for some rational prime $p$. Equivalently, the completion $K_\upsilon$ of $K$ with respect to $\upsilon$ is a finite extension of the field $\Q_p$ of $p$-adic numbers. On the other hand, an embedding of $K$ into the field $\C$ of complex numbers is an \textit{infinite} place of $K$. If the image of $K$ under this embedding is real, the corresponding place is dubbed a \textit{real} place; and if not, it is \textit{complex}. The image of $K$ is then $\R$ or $\C$, respectively. \par 
The \textit{adele ring} $\A_K$ of $K$ is the locally compact topological ring given by the restricted direct product
\begin{align*}
\A_K = \prod_{\upsilon \mid \infty} K_\upsilon \times \prodprime_{\upsilon < \infty} K_\upsilon = \lim_{\rightarrow} \A_S, \quad \quad \A_S = \prod_{\upsilon \in S} K_\upsilon \times \prod_{\upsilon \not \in S} \ringint_\upsilon,
\end{align*}
where $S$ runs over all finite subsets of places containing all the infinite places. We recall that the restricted direct product defining $\A_K$ is characterized by the fact that, as far as the coordinates belonging to the finite places are concerned, at most finitely many coordinates lie outside of the ring $\ringint_\upsilon := \ringint_{K_\upsilon}$ of integers in $K_\upsilon$. In particular, the number field $K$ embeds diagonally into $\A_K$. We refer to \cite[Chapter 1]{platrap} for more information.
\subsection{Algebraic groups}\label{subsectionalgebraicgroups}
We will now recall some standard definitions and results from the theory of linear algebraic groups (cf. also \cite[Chapters 2 and 3]{platrap}). Throughout this section, $K$ will denote a number field, and $\G$ will denote a linear algebraic group. \par 
A \textit{(linear) algebraic group} is a Zariski-closed subgroup $\G \subseteq \SL_n (\C)$ of some special linear group over the complex numbers. In other words, $\G$ can be identified with the vanishing locus of an ideal of polynomials. We say that $\G$ is \textit{defined over} $K$, or that $\G$ is an \textit{$K$-group}, if $\G$ can be defined by a set of polynomials over $K$.\par 
We say that an algebraic group $\G$ is \textit{simply connected} if any isogeny from a connected algebraic group to $\G$ is trivial. Moreover, we will call $\G$ \textit{$K$-simple} if it does not contain any non-trivial connected, closed, normal subgroups defined over $K$. In case $\G$ is defined over $K$ and simple over an algebraic closure $\overline{K}$ of $K$, we will call $\G$ \textit{absolutely simple}. \\ \par 
For $K$-groups $\G$ and $\mathbb{H}$, we say that $\Hh$ is a \textit{$K$-form} of $\G$ if there exists an isomorphism from $\G$ to $\Hh$ defined over $\overline{K}$. Two forms are said to be \textit{equivalent} if they are isomorphic over $K$. Let $E(K, \G)$ be the set of all equivalence classes of $K$-forms of $\G$. The Galois group $\mathrm{Gal} (\overline{K}/K)$ acts on the set of all isomorphisms from $\G$ to $\Hh$: If $\varphi : \G \rightarrow \Hh$ is an isomorphism and $\sigma \in \mathrm{Gal}(\overline{K}/K)$, then $\sigma. \varphi := \varphi^\sigma$ is the isomorphism obtained by applying $\sigma$ to all the coefficients of the rational functions defining $\varphi$. Let now $\varphi$ be any fixed isomorphism from $\G$ to $\Hh$. Then we obtain a map $\mathrm{Gal}(\overline{K}/K) \rightarrow \mathrm{Aut}_{\overline{K}} (\G)$ to the group of all $\overline{K}$-automorphisms of $\G$, which is given by $\sigma \mapsto \varphi^{-1} \circ \varphi^{\sigma}$. In the case when the image of this map in $\mathrm{Aut}_{\overline{K}}(\G)$ consists only of inner automorphisms of $\G$ (considered as an algebraic group over $\overline{K}$), we say that $\Hh$ is an \textit{inner form} of $\G$. If not, $\Hh$ is called an \textit{outer form} of $\G$.\par 
The group $\SP_\yoyo (\R)$ has no outer forms. Indeed, from \cite[Sect. 2.1.13]{platrap} it follows in particular that the number of isomorphism classes of outer forms of a group is equal to the number of graph automorphisms of its Dynkin diagram. Since the diagram $C_n$ has no non-trivial automorphisms, our claim follows. 
\par On the other hand, the (simply connected) inner forms of $\SP_\yoyo (\R)$ can be described in the following straightforward manner: By \cite[Prop. 2.19]{platrap} every simply connected inner form of a group of type $C_n$ is of the form $\SU_m (D, f)$ where $D$ is a central division algebra of index $2n/m \in \Z$ over $K$, equipped with an involutive antiautomorphism $\tau : D \rightarrow D$, and $f$ is a nondegenerate sesquilinear form, which is either Hermitian or skew-Hermitian (depending on the \textit{type} of $\tau$ --- see \cite[Sect. 2]{platrap}).\\ \par 
A subgroup $B \subseteq \G$ is called a \textit{Borel subgroup} if it is both connected and solvable and not properly contained in any other connected, solvable subgroup of $\G$. $\G$ itself is called \textit{quasisplit} over $K$ if it contains a Borel subgroup which is defined over $K$. An important fact is that over any field $K$, an algebraic group has a unique quasisplit inner form (cf. \cite[Prop. 7.2.12]{brianconrad}).\par  
Another distinguished type of subgroup $T \subseteq \G$ is called a \textit{torus} if it is a connected and closed subgroup of $\G$ which is diagonalizable over an algebraic closure of $K$. We say a torus is \textit{$K$-split} if is diagonalizable over $K$, and the group $\G$ is called \textit{$K$-split} if it contains a maximal $K$-split torus which is defined over $K$.\par It is well-known that if a group $\G$ is $K$-split, it is also quasisplit over $K$. In particular, for an algebraic group without outer forms, the notions of a quasisplit form and a split form coincide. \\ \par  
Let $\G$ be an algebraic group defined over $K$. Then there is a canonical way to turn $\G$ into an algebraic group defined over $\Q$, namely \textit{restriction of scalars} or \textit{Weil restriction}. \par 
Suppose that the $K$-group $\G$ equals the vanishing locus of the ideal $\langle f_1, \ldots, f_m \rangle$, where each $f_j$ is a polynomial with coefficients in $K$. Let $r_1$ and $r_2$ be the respective number of real and complex embeddings of $K$. For any such embedding $\sigma : K \hookrightarrow \C$, we then let $\G^\sigma$ be the group whose corresponding ideal is $I(\G^\sigma) = \langle \sigma \circ f_1, \ldots, \sigma \circ f_m \rangle$. Then the \textit{restriction of scalars} $\mathrm{Res}_{K/\Q}(\G)$ of $\G$ is the algebraic $\Q$-group defined by
\begin{align*}
    \mathrm{Res}_{K/\Q}(\G) = \prod_{i = 1}^{r_1 + r_2} \G^{\sigma_i}.
\end{align*}
(Strictly speaking, $\mathrm{Res}_{K/\Q}(\G)$ is an algebraic group which is isomorphic over $\overline{K}$ to the product on the right-hand side above. More specifically, the group of $\Q$-rational points of $\mathrm{Res}_{K/\Q}(\G)$ is isomorphic to the group of $K$-rational points of $\G$.) In particular, we take note of the fact that
\begin{align*}
\mathrm{Res}_{K/\Q}(\G)(\R) \simeq \prod_{\upsilon \mid \infty} \G (K_\upsilon),
\end{align*}
where $K_\upsilon$ is the completion of $K$ at the infinite place $\upsilon$. For more details, we refer to \cite[Sect. 2.1.2]{platrap}. \\ \par 
Rather than looking at rational points over a smaller field, we can also consider the group of points of $\G$ over the adele ring $\A_K$. Since the ring operations work coordinatewise, one has the explicit description
\begin{align*}
\G (\A_K ) = \prodprime_{ \upsilon } \G (K_\upsilon) = \G_\infty \times \prodprime_{\upsilon < \infty} \G (K_\upsilon).
\end{align*}
Furthermore, as a consequence of the fact that $K$ embeds diagonally into $\A_K$, the same is true for the group $\G (K)$ of $K$-rational points, which in fact embeds as a discrete subgroup of $\G (\A_K)$. We refer to \cite[Sect. I.(0.33)]{margulis} for more details. \\ \par 
As mentioned above, we know that our lattice $\Gamma$ is a maximal arithmetic subgroup of $\SP_\yoyo (\R)$. In consequence, $K$ must be a totally real number field: \par
First of all, we see from \cite[§18.5]{morris} that at least $K \subseteq \R$. Next, \cite[Corollary 5.5.16]{morris} shows that $\G (\R)$ is simple, and that the diagonal embedding $\Delta (\G (\ringint_K))$ in $\mathrm{Res}_{K/\Q}(\G)(\R)$ is therefore an \textit{irreducible} lattice due to \cite[Prop. 5.5.8 and Remark 5.5.9]{morris}. \par If we consider the projection of $\mathrm{Res}_{K/\Q}(\G)(\R)$ onto the factor corresponding to the identity embedding, we have $\pi (\Delta (\G (\ringint_K))) = \G(\ringint_K) = \Gamma$. Accordingly, if the product defining the restriction were to contain two or more non-compact factors, the projection of an irreducible lattice would (by definition) be dense in the first factor. Since $\Gamma$ is discrete, it therefore follows that $\G (\overline{\sigma(K)})$ is compact for any embedding $\sigma \neq \mathrm{id}$. \par This means, in turn, that any such $\sigma$ must be a real embedding of $K$ since, in the alternative case where $\sigma (K) \not \subseteq \R$, the group $\G  (\overline{\sigma ( K )} ) \simeq \SP_\yoyo (\C)$ is not compact. (By contrast, the explicit description of $\G$ as a special unitary group shows that for a real embedding $\sigma$, the corresponding group of points can be compact.)
\subsection{Prasad's volume formula}\label{subsectionPFV}
One of the most subtle parts of Prasad's volume formula has to do with a certain set of compact open subgroups which are intrinsically linked to the principal arithmetic subgroup under consideration. Before discussing the volume formula in detail, we will describe some generalities about these groups.
\subsubsection{Parahoric subgroups}\label{subsubsectionparahoricsubgroups}
Just as one has the notions of \textit{Borel} and \textit{parabolic} subgroups of a reductive algebraic group, one has the notions of \textit{Iwahori} and \textit{parahoric} subgroups of an algebraic group over non-archimedean local fields. \par 
As a matter of fact, the following picture provides a useful intuition: if $K_\upsilon$ is the completion of a number field $K$ at a finite place $\upsilon$, and $\ringint_\upsilon \subseteq K_\upsilon$ is the valuation ring of $K_\upsilon$ containing the unique maximal ideal $\mathfrak{m}_\upsilon = \lbr x \in \ringint_\upsilon : \upsilon (x) > 0 \rbr$, one has the usual projection map $\ringint_\upsilon / \mathfrak{m}_\upsilon \twoheadrightarrow \F_{q_\upsilon}$ to the finite field with $q_\upsilon$ elements. An Iwahori subgroup $I_\upsilon \subseteq \G (K_\upsilon)$ is then essentially the preimage $\pi^{-1} (B)$ of a Borel subgroup $B$ of $\overline{\G}$, the group $\G$ “considered” as a group over the finite field $\F_{q_\upsilon}$. (This will be made precise below.) Furthermore, a parabolic subgroup of $\overline{\G} (\F_{q_\upsilon})$ is, by definition, any group containing (a conjugate of) $B$, and a parahoric subgroup $P_\upsilon \subseteq \G (K_\upsilon)$ is then essentially the preimage $\pi^{-1}(P)$ of a parabolic subgroup. \\ \par 
We now proceed to more precise definitions. A subgroup $B_\upsilon \subseteq \G (K_\upsilon)$ is an \textit{Iwahori subgroup} if it is the normalizer of a maximal pro-$p$-subgroup of $\G (K_\upsilon)$ (that is, the inverse limit of a coherent sequence of finite $p$-groups, which is not properly contained in any other such group). A subgroup $P_\upsilon \subseteq \G (K_\upsilon)$ is called \textit{parahoric} if it contains an Iwahori subgroup.\par
We recall the following facts from \cite[Sect. 3.4]{platrap}: There is an Iwahori subgroup $B_\upsilon$ in $\G (K_\upsilon)$ and a maximal $K_\upsilon$-split torus $S_\upsilon \subseteq \G(K_\upsilon)$ (say, of dimension $\dim S_\upsilon = \ell$) such that with $N_\upsilon = (N_{\G}(S_\upsilon))(K_\upsilon)$ equal to the normalizer of $S_\upsilon$, the pair $(B_\upsilon,N_\upsilon)$ is a so-called \textit{$BN$-pair} for the group of $K_\upsilon$-rational points $\G (K_\upsilon)$. This means in particular that $H_\upsilon := B_\upsilon \cap N_\upsilon $ is a normal subgroup in $N_\upsilon$, and that there is a size $\ell + 1$ generating subset $\Delta_\upsilon = \lbr r_0, r_1, \ldots, r_\ell \rbr \subseteq W_\upsilon := N_\upsilon / H_\upsilon$ of the Weyl group $W_\upsilon$, corresponding to the vertices in the local Dynkin diagram of $\G(K_\upsilon)$, with every element of $\Delta_\upsilon$ having order $2$. It turns out that if $B_\upsilon \subseteq P_\upsilon$ for some subgroup $P_\upsilon \subseteq \G (K_\upsilon)$, then there is a subset $\Theta_\upsilon \subseteq \Delta_\upsilon$ of generators, and a resulting subgroup $W_{\Theta_\upsilon} \subseteq W_{\Delta_\upsilon} = W_\upsilon$ generated by the elements of $\Theta_\upsilon$, such that $P_\upsilon = B_\upsilon W_{\Theta_\upsilon} B_\upsilon$. The subset $\Theta_\upsilon \subseteq \Delta_\upsilon$ is then called the \textit{type} of $P_\upsilon$. 
\par From the perspective of \textit{Bruhat--Tits buildings}, a subgroup $P = P_\upsilon \subseteq \G (K_\upsilon)$ is parahoric if and only if it is the stabilizer of a simplex in the Bruhat--Tits building $\mathcal{B}(\G, K_\upsilon)$ associated with $\G (K_\upsilon)$. Naturally, the correspondence between simplices and their stabilizers is inclusion-reversing. In particular, the maximal parahoric subgroups of $\G (K_\upsilon)$ are precisely the stabilizers of individual vertices in the building $\mathcal{B}(\G, K_\upsilon)$. \par 
A particular class of parahoric subgroups will play a central role in our subsequent analysis, namely the special and hyperspecial parahorics. We say that a parahoric subgroup $P_\upsilon \subseteq \G (K_\upsilon)$ is \textit{(hyper)special} if it stabilizes a point in $\mathcal{B}(\G, K_\upsilon)$, which is \textit{(hyper)special}. Here, a point $x \in \mathcal{A} = \mathcal{A}(\G)$ in an apartment of the building is \textit{special} if every hyperplane in $\mathcal{A}$ is parallel to a hyperplane that passes through $x$ (see \cite[Prop. 10.19]{abramenkobrown}). Moreover, such a point $x$ is called \textit{hyperspecial} if it continues to be special in the Bruhat--Tits building $\mathcal{B} (\G, \widehat{K_\upsilon})$ where $\widehat{K_\upsilon}$ denotes the maximal unramified extension of $K_\upsilon$. (Note that all apartments in a building are isometric. These properties therefore do not depend on the choice of an apartment containing the given point.) \par For example, if $\G = \SP_{2n}$ is defined over $\Q$, and $\upsilon$ is a finite place corresponding to the prime number $p$, the subgroup $\SP_{2n}(\Z_p) \subseteq \SP_{2n}(\Q_p)$ is a hyperspecial parahoric subgroup.\\\\
{\textit{Remarks.}
\par \textbf{1)} We recall from \cite[Sect. 3.1]{borelprasad} that a parahoric subgroup $P_\upsilon$, which has maximal volume among all parahoric subgroups, is necessarily special; and from \cite[Sect. 3.2]{borelprasad} that a hyperspecial parahoric subgroup necessarily has maximal volume. (See also \cite[Prop. A.5]{borelprasad}.)
\par \textbf{2)} Since we want $\Gamma \subseteq \SP_{2n} (\R)$ to be of minimal covolume, we claim that we are, in fact, free to assume that at each finite place $\upsilon$, the parahoric subgroup $P_\upsilon \subseteq \G (K_\upsilon)$ has maximal volume. Namely, this follows from the following two points: First of all, the argument in \cite[Sect. 4.3]{belolipetskyemery} (which builds on \cite[Sect. 3.8]{borelprasad}) shows that the covolume of $\Gamma = N_{\G (\R)} (\Lambda)$ can only \textit{decrease} when a parahoric subgroup $P_\upsilon$ in the coherent sequence of $\Lambda$ is replaced with a parahoric of larger volume, in case $P_\upsilon$ is not already maximal. (Although the statements made in \cite{belolipetskyemery} pertain to groups of type $D_n$, the proof of the inequality after \cite[eq. 4.(15)]{belolipetskyemery} extends almost immediately to groups of type $C_n$ as well, with only minor adjustments required.) Second, we will show that the covolume \textit{must, in fact, decrease} in this situation. This argument will rely on the analysis that follows (where every $P_\upsilon$ is assumed to be maximal) and, in particular, on Theorem \ref{prasadinourspecialcase} below, which we have not yet discussed. For this reason, we postpone the justification of this claim to Section \ref{sectionlocalcontributions} where the result is stated in the form of Lemma \ref{newlemma_covolume_must_decrease}.  \\ 
\par Assume from now on that each $P_\upsilon$ is special.} This fact has at least the following useful consequences. First of all, the type $\Theta_\upsilon$ of any parahoric $P_\upsilon$ in our sequence has no symmetries, as is shown in \cite[Sect. 3.1]{emerykim}. Second, the following simple characterization of hyperspecial parahoric subgroups holds: for a finite place $\upsilon$,
\begin{align}\label{09012025-convenient}
\text{$P_\upsilon$ is hyperspecial } \Longleftrightarrow \text{ $\G$ splits over $K_\upsilon$.}
\end{align}
Indeed, it follows immediately from the definition that $P_\upsilon$ is hyperspecial if $\G$ splits over $K_\upsilon$. Conversely, if $P_\upsilon \subseteq \G (K_\upsilon)$ is hyperspecial, then $\G$ must be quasisplit over $K_\upsilon$ by \cite[Prop. 10.2.1]{kalethaprasad}. Accordingly, $\G$ splits over $K_\upsilon$ by uniqueness of the quasisplit inner form over any field. \\ \par 
We now proceed to the discussion of parahoric subgroups in the context of Prasad's volume formula. Additional details are given in \cite[Sect. 2.2]{prasad} and \cite[Sect. 4.1]{kalethaprasad}. \par 
A parahoric subgroup $P_\upsilon$ is related to a smooth affine $\ringint_\upsilon$-group scheme $\textbf{G}_\upsilon$. In particular, the group $P_\upsilon$ coincides with the $\ringint_\upsilon$-points of this scheme, i.e. $\textbf{G}_\upsilon ( \ringint_\upsilon) = P_\upsilon$. Therefore one may consider the reduction (mod $\mathfrak{m}_\upsilon$) of $P_\upsilon$ in the form of the base change $\overline{\textbf{G}}_\upsilon := \textbf{G}_\upsilon \times_{\ringint_\upsilon} \F_{q_\upsilon}$. (Intuitively speaking, the resulting group of points simply consists of the points of $\textbf{G}_\upsilon$ where all of its defining equations have been reduced modulo $q_\upsilon$, over the finite field $\F_{q_\upsilon}$.) The group $\overline{\textbf{G}}_\upsilon (\F_{q_\upsilon})$ then admits a Levi decomposition $\textbf{M}_\upsilon \ltimes \mathrm{R}_u (\overline{\textbf{G}}_\upsilon (\F_{q_\upsilon}))$ where $\textbf{M}_\upsilon \subseteq \overline{\textbf{G}}_\upsilon (\F_{q_\upsilon})$ is a maximal connected reductive subgroup (the \textit{Levi component}), and $\mathrm{R}_u (\overline{\textbf{G}}_\upsilon (\F_{q_\upsilon}))$ denotes the unipotent radical. \par 
If we now let $\G$ be a $K$-form of $\SP_\yoyo$ and let $\mathcal{G}$ be the unique quasisplit inner $K$-form of $\G$, we obtain in an analogous way the maximal connected reductive subgroup $\mathcal{M}_\upsilon \subseteq \overline{\mathcal{G}}_\upsilon (\F_{q_\upsilon})$. Since $\SP_\yoyo$ is itself split, we have in fact $\mathcal{G} = \SP_\yoyo$ and $\mathcal{M}_\upsilon = \SP_\yoyo (\F_{q_\upsilon})$.
\par With these details in place, we can now describe the local factors appearing in Prasad's volume formula. Thus, we define
\begin{align}\label{eNOstrich}
e (P_\upsilon) := \frac{q_{v}^{(\dim \textbf{M}_\upsilon \, + \, \dim \mathcal{M}_\upsilon)/2}}{\# \textbf{M}_\upsilon (\F_{q_\upsilon})},
\end{align}
and 
\begin{align}\label{estrich} 
e' (P_\upsilon) := q_\upsilon^{(\dim \textbf{M}_\upsilon - \dim \mathcal{M}_\upsilon)/2} \frac{\# \mathcal{M}_\upsilon (\F_{q_\upsilon})}{\# \textbf{M}_\upsilon (\F_{q_\upsilon})} = e(P_\upsilon) \frac{\# \mathcal{M}_\upsilon (\F_{q_\upsilon})}{q_\upsilon^{\dim \mathcal{M}_\upsilon}}.
\end{align}
We note that $e' (P_\upsilon)$ is always a non-negative integer (cf. \cite[Sect. 2.5]{prasadyeung}). \par 
Of course, it is possible to simplify both of these expressions to some extent by using the explicit description of $\mathcal{M}_\upsilon$ as the standard split form of the symplectic group over the residue field $\F_{q_\upsilon}$. We will do this in the next section. For now, we note (cf. \cite[Sect. 2.5]{prasadyeung}) that one always has the inequality
\begin{align}\label{evsestrich}
e'(P_\upsilon) < e(P_\upsilon).
\end{align}
For later use, we also note the following equivalences (cf. \cite[Sect. 2.2]{prasad}):
\begin{align*}
    e'(P_\upsilon) = 1 &\Longleftrightarrow  \dim \overline{M}_\upsilon = \dim \overline{\mathcal{M}}_\upsilon \text{ and } \# \overline{M}_\upsilon (\mathfrak{f}_\upsilon) = \# \overline{\mathcal{M}}_\upsilon ( \mathfrak{f}_\upsilon) \\ &\Longleftrightarrow \text{$P_\upsilon$ is hyperspecial.}
\end{align*}
\subsubsection{The statement of the volume formula}
With the necessary preparations in place, we will now describe the main technical tool used in the paper, which is Prasad's voluma formula. In our specific case of interest where $\G$ has type $C_n$ and $K$ is totally real (see the end of Section \ref{subsectionalgebraicgroups}), this formula can be stated as follows. For both this result and future purposes, it will be useful to define (as we did implicitly in the introduction) the number
\begin{align}\label{piofnisgivenby}
    \Pi (n) := \prod_{j = 1}^n \frac{(2j-1)!}{(2 \pi)^{2j}}.
\end{align}
\begin{theorem}[{\cite[Thm. 3.7]{prasad}}]\label{prasadinourspecialcase}
Let $\G$ be an absolutely simple, simply connected algebraic group of rank $n$ and type $C_n$ defined over a totally real number field $K$. Let $\mu_\infty$ denote the product measure on $\G_\infty$, where (for each infinite place $\upsilon$) the measure $\mu_\upsilon$ on $\G (K_\upsilon)$ is described in \cite[Sect. 1.3, Sect. 1.4]{prasad}. Let $\Lambda$ be the principal arithmetic subgroup determined by a coherent collection $\lbr P_\upsilon : \upsilon < \infty \rbr$ of parahoric subgroups $P_\upsilon \subseteq \G (K_\upsilon)$. Then we have the formula
\begin{align*}
\mu_\infty (\G_\infty / \Lambda ) = D_K^{n(2n+1)/2} \Pi (n)^{d_K} \! \prod_{j = 1}^{n} \zeta_K (2j) \prod_{\upsilon < \infty} e'(P_\upsilon),
\end{align*}
where the factors $e'(P_\upsilon)$ are given by \eqref{estrich}, and $\zeta_K$ denotes the Dedekind zeta function of $K$.
\end{theorem}
\noindent \textit{Remark.} The measure $\mu_\infty$ on $\G_\infty$ is normalized to give any maximal compact subgroup measure $1$.  In particular, if $\G (K_\upsilon)$ is compact for any infinite place $\upsilon \neq \upsilon_0$, the left-hand side of Theorem \ref{prasadinourspecialcase} is $\mu (\G (K_{\upsilon_0}) / \Lambda )$.
\begin{proof}
Let $\upsilon$ be any finite place of $K$. Given that the exponents $m_i$ of $\G$ are $m_i = 2i-1$ for $i = 1, \ldots, n$ (see \cite[Sect. 2.4]{prasadyeung}), we observe that
\begin{align*}
\dim \G = n + 2(1 + 3 + \cdots + 2n-1) = 2n^2 + n,
\end{align*}
and, to justify the appearance of the Dedekind zeta values, that
\begin{align}\label{campusfestival}
e' (P_\upsilon) = e( P_\upsilon) \prod_{j = 1}^n \lp 1 - \frac{1}{q_\upsilon^{m_j + 1}} \rp = e(P_\upsilon) \prod_{j = 1}^n \lp 1 - \frac{1}{q_\upsilon^{2j}} \rp,
\end{align}
where $q_\upsilon = \# \mathfrak{f}_\upsilon = N(\mathfrak{p}_\upsilon)$ denotes the size of the residue field $\mathfrak{f}_\upsilon$ of $K$ at $\upsilon$ (or, equivalently, the norm of the prime ideal $\mathfrak{p}_\upsilon \subseteq \ringint_{K}$ associated with $\upsilon$). Indeed, this follows from \cite[Sect. 2.4]{prasadyeung} and the fact that $\G$ is of type $C_n$ and thus without outer forms. Upon multiplying the right-hand side above over all finite places $\upsilon$ of $K$, one recognizes the reciprocal of the Euler product defining $\zeta_K (2) \cdots \zeta_K (2n)$.\par 
As noted above, the quasisplit inner $K$-form $\mathcal{G}$ of $\G$ splits over $K$. Therefore $L = K$, and the factor involving $D_L / D_K^{\lb L:K \rb}$ vanishes. Moreover, since the number field $K$ is totally real, its number of (inequivalent) infinite places equals $\lb K : \Q \rb = d_K$, and hence the product over the infinite places of $K$ can be rewritten as claimed (cf. also \cite[Rem. 3.8]{prasad}). The claimed formula now follows from these observations and the fact that the Tamagawa number $\tau_K (\G)$ is equal to $1$ (see \cite{kottwitz-tamagawa}).
\end{proof}
Since our approach to using Prasad's Volume Formula involves the initial step of ignoring the contribution of the parahoric factors $e'(P_\upsilon)$ to the covolume $\mu (\G (K_{\upsilon_0} / \Lambda)$, it will be convenient for us to introduce the notation
\begin{align}\label{covolumewithoutparahorics}
S(\Lambda) := \mu (\G (K_{\upsilon_0}) / \Lambda ) \lp \prod_{\upsilon < \infty} e'(P_\upsilon) \rp^{-1} = D_K^{n(2n+1)/2} \Pi (n)^{d_K} \! \prod_{j = 1}^{n} \zeta_K (2j).
\end{align}
In terms of this number, the covolume of $\Gamma$ can be expressed as
\begin{align}\label{08012025-gammacovolume}
\mu \lp \G (K_{\upsilon_0}) / \Gamma \rp = \frac{1}{\lb \Gamma : \Lambda \rb} S(\Lambda)\prod_{\upsilon < \infty} e' (P_\upsilon).
\end{align}
\section{Bounding the Covolume: Global considerations}\label{SectionBoundingthecovol}
In this section we estimate the right-hand side of Prasad's volume formula from below and to different degrees of accuracy. As we described in the introduction, the assumption that $\Gamma$ has smaller covolume than $\Gamma_0$ will, in combination with the simplest bounds we prove, result in a number of numerical bounds on the arithmetic invariants related to our number field, e.g. $D_K$ and $d_K$. Since many number fields are thus excluded from our list of candidates, we can use progressively finer (and more involved) bounds to rule out even more number fields and trim the list even further.
\subsection{Lower bounds in terms of $d_K$, $D_K$, and $n$}
By Theorem \ref{prasadinourspecialcase} and the discussion at the end of Section \ref{subsubsectionparahoricsubgroups}, we have the lower bound $\mu (\G (K_{\upsilon_0}) / \Lambda ) \geqslant S(\Lambda)$. We begin this subsection by converting this fact into a lower bound on the covolume of $\Gamma$, making use of the different number theoretic estimates mentioned in Subsection \ref{subsectionnumberfields}. 
\begin{lemma}\label{covolprotobound}
    Let $K$ be a totally real number field of degree $\lb K : \Q \rb = d_K$ and class number $h_K$, and let $D_K$ be the absolute value of the discriminant of $K$ over $\Q$. Then the covolume of $\Gamma$ satisfies
    \begin{align*}
        \mu (\G (K_{\upsilon_0}) / \Gamma) \geqslant \frac{S(\Lambda)}{\lb U_K^+ : U_K^2 \rb \cdot h_K} \geqslant \frac{S(\Lambda)}{2^{2d_K-1} h_K},
    \end{align*}
where the subgroups $U_K^+, \, U_K^2 \subseteq \ringint_K^\times$ are defined in Subsection \ref{subsectionnumberfields}.
\end{lemma}
\begin{proof}
    Let $\mathcal{T}$ denote the set of finite places $\upsilon$ of $K$ with the property that $P_\upsilon$ is \textit{not} hyperspecial. (Equivalently, by \eqref{09012025-convenient}, $\mathcal{T}$ is the set of all $\upsilon$ such that $\G$ does not split over $K_\upsilon$.) From the discussion in Section \ref{subsubsectionparahoricsubgroups} we know that the types $\Theta_\upsilon$ of the parahorics $P_\upsilon$ are without symmetries. By \cite[Lemma 6.3]{emerykim} we then have the bound
    \begin{align}\label{indexboundgaloiscohomology}
    \lb \Gamma : \Lambda \rb \leqslant h_K 2^{\# \mathcal{T}} \lb U_K^+ : U_K^2 \rb \leqslant h_K 2^{2d_K - 1 + \# \mathcal{T}},
    \end{align}
    where the last inequality is \eqref{totallypositivemodsquares}. To derive our claim from Theorem \ref{prasadinourspecialcase}, we now write
    \begin{align*}
    \mu (\G (K_{\upsilon_0}) / \Lambda ) = S(\Lambda ) \, \prod_{\upsilon \in \mathcal{T}} e'(P_\upsilon) \prod_{\upsilon \not \in \mathcal{T}} e'(P_\upsilon).
    \end{align*}
    From our discussion in Section \ref{subsubsectionparahoricsubgroups} we also see that the product over $\upsilon \not \in \mathcal{T}$ equals $1$, and that each factor in the other product is an integer $e'(P_\upsilon)$ strictly greater than $1$. In combination with the first inequality in \eqref{indexboundgaloiscohomology}, this shows that
    \begin{align*}
    \mu (\G (K_{\upsilon_0}) / \Gamma ) &= \mu (\G (K_{\upsilon_0}) / \Lambda ) \lb \Gamma : \Lambda \rb^{-1} \\
    &\geqslant S(\Lambda) 2^{\# \mathcal{T}} h_K^{-1} 2^{- \# \mathcal{T}} \lb U_K^+ : U_K^2 \rb^{-1} \\ 
    &= S(\Lambda) h_K^{-1} \lb U_K^+ : U_K^2 \rb^{-1}.
    \end{align*}
    This proves the first bound. The second bound is proved analogously with the second inequality in \eqref{indexboundgaloiscohomology}.
\end{proof}
\noindent \textit{Remarks.} 
\par \textbf{1)} We should point out that, although the result \cite[Lemma 6.3]{emerykim} is stated and proved in the context of a different real form of the symplectic group, its proof makes no special use of the structure of ${\bf G}$ at the real places of the relevant number field $k$ (in the notation of \cite{emerykim}). Rather, the argument only relies on the structure at the non-archimedean places. As such, all types of parahorics in $\G (K_\upsilon)$ can occur within the framework of \cite[Lemma 6.3]{emerykim}, and its analysis is therefore exhaustive in our situation as well.
\par \textbf{2)} In order to keep things as simple as possible, we will mainly use the weaker of the two estimates in Lemma \ref{covolprotobound} in the sequel, as it has the benefit that its right-hand side depends only on $d_K$ and $h_K$ and not on the unit groups. On the other hand, since the index $\lb U_K^+ : U_K^2 \rb$ can be computed rather easily for many particular number fields, we will make use of the full strength of the lemma once our other analyses have identified a small list of candidates for $K$.
\par {\textbf{3)} For the convenience of readers unfamiliar with Galois cohomology, we wish to mention the paper \cite{newallan} where a version of the bound \eqref{indexboundgaloiscohomology} is derived in the special case $h_K = 1$ using other techniques than those in \cite{emerykim} and \cite{borelprasad}. In particular, it is demonstrated here how a change in the involved parahoric subgroups can change the index $\lb \Gamma : \Lambda \rb$.}

\begin{corollary}\label{protodbound}
    Suppose that $\mu \lp \G (K_{\upsilon_0}) / \Gamma \rp \leqslant \Psi(n)$ where $\Psi(n)$ is defined in \eqref{referencecovolume}. Then the discriminant of $K$ satisfies the bound
    \begin{align}\label{protoD}
     D_K < \lp 0.915 \cdot 2^{2d_K} h_K \Pi(n)^{1-d_K} \rp^{1/(n^2 + n/2)}.
\end{align}
\end{corollary}
\begin{proof}
    This follows from straightforward manipulations of the weaker estimate given in the lemma. We only need the additional observation that
    \begin{align}\label{09012025-numericalupperbound}
        \Psi(n) < \Pi (n) \prod_{j = 1}^\infty \zeta(2j) < 1.83 \cdot \Pi(n),
    \end{align}
where the numerical bound for the infinite product follows from \cite[Lemma 1]{kellner-zeta}.
\end{proof}
\begin{lemma}\label{LLLemmaxyx}
    The covolume $\mu \lp \G (K_{\upsilon_0}) / \Gamma \rp$ satisfies the bound
\begin{align}
 \mu \lp \G (K_{\upsilon_0}) / \Gamma \rp > F (d_K, D_K, n) := \frac{1}{750} D_K^{n^2 + n/2 - 3} \lp 7.6  e^{0.46} \Pi(n) \rp^{d_K}. \label{covolest04}
\end{align}
\end{lemma}
\begin{proof}
If $R_K$ is the regulator of $K$ and $t > 0$ is any real number, Proposition \ref{brauersiegel} and Zimmert's bound for the regulator \eqref{zimmert} imply that
\begin{align*}
\frac{1}{h_K} \geqslant \frac{R_K}{H(d_K, D_K, t)} \geqslant \frac{1}{25 \cdot H(d_K, D_K, t)} e^{0.46 \cdot d_K} ,
\end{align*}
where
\begin{align}\label{thefunctionH}
H(d_K, D_K, t) = 2 t (t+1) \lp \frac{1}{2 \pi^{(1+t)/2}} \Gamma \lp \frac{t+1}{2} \rp \rp^{d_K} D_K^{(t+1)/2} \zeta_K (t+1).
\end{align}
With Lemma \ref{covolprotobound} we thus obtain the estimate
\begin{align}\label{covolest02}
\nonumber \mu \lp \G (K_{\upsilon_0}) / \Gamma \rp &\geqslant \frac{2}{25} D_K^{n(2n+1)/2} \lp \frac{1}{4} e^{0.46} \Pi (n) \rp^{d_K} \frac{1}{H(d_K, D_K, t)} \\
\nonumber &= \frac{1}{25} D_K^{n(2n+1)/2} \lp \frac{1}{2} e^{0.46} \pi^{(t+1)/2} \Gamma \lp \frac{t+1}{2} \rp^{\! \! -1} \Pi (n) \rp^{d_K} \\ 
\nonumber &\quad \quad \times \frac{1}{t (t+1) D_K^{(1+t)/2} \zeta_K (t+1)}\\[+1em]
&\geqslant \frac{D_K^{(2n^2 + n - t - 1)/2}}{25 t (t+1)} \lp \frac{1}{2} e^{0.46} \Pi(n) \alpha(t+1) \rp^{d_K},
\end{align}
where we used the classical estimate $\zeta_K (t+1) \leqslant \zeta(t+1)^{d_K}$ and wrote
\begin{align*}
\alpha (t+1) = \pi^{(t+1)/2} \Gamma \lp \frac{t+1}{2} \rp^{\! \! -1} \zeta(t+1)^{-1} = \frac{t (t+1)}{2 \xi (t+1)},
\end{align*}
where $\xi$ denotes the \textit{Riemann xi function}. Since we are free to choose the value of $t$, we take $t = 4.99$ which gives $\alpha (t+1) \approx 15.2199$. Then since $t(t+1) < 30$, we conclude from \eqref{covolest02} that
\begin{align*}
\nonumber \mu \lp \G (K_{\upsilon_0}) / \Gamma \rp > \frac{1}{750} D_K^{n^2 + n/2 - 3} \lp 7.6  e^{0.46} \Pi(n) \rp^{d_K}.
\end{align*}
This completes the proof.
\end{proof}
\begin{corollary}\label{toyotacorollary}
    Suppose that $\mu \lp \G (K_{\upsilon_0}) / \Gamma \rp \leqslant \Psi(n)$. Then the discriminant of $K$ satisfies the bound
    \begin{align}\label{covolestD}
     D_K < \lp 1372.5 \cdot \Pi (n)^{1-d_K}  \Big( 7.6 e^{0.46} \Big)^{-d_K} \rp^{1/(n^2 + n/2 - 3)}.
\end{align}
\end{corollary}
\begin{proof}
    This follows in an analogous way to the proof of Corollary \ref{protodbound}.
\end{proof}
\subsection{Monotonicity of the lower bounds}
In this section we arrive at the following preliminary conclusions: if $n \geqslant 3$, then $K = \Q$; and if $n = 2$, then $K$ is either $\Q(\sqrt{5})$ or $\Q$.\par 
We proceed by bounding the right-hand side of \eqref{covolest04} in terms of only $d_K$ and $n$, effectively eliminating $D_K$ from the statement of Lemma \ref{LLLemmaxyx} (for now). This will allow us to deduce upper bounds on $d_K$ (depending on $n$) under the assumption that $\Gamma$ has smaller covolume than $\Gamma_0$.
\subsubsection*{I. The cases $\bm{ n \geqslant 4}$}
We proceed by bounding $D_K$ from below in terms of $d_K$ with Odlyzko's estimates. In our case where $K$ is totally real, inserting the bound \eqref{odlyzkobounds} into the bound \eqref{covolest04} from Lemma \ref{LLLemmaxyx} and writing $f(n) = n^2 + n/2 - 3$, we obtain that
\begin{align}\label{odlyzkocovolbound}
\mu \lp \G (K_{\upsilon_0}) / \Gamma \rp > O(n,d_K, A,E),
\end{align}
where
\begin{align*}
    O(n,d_K, A,E) := F(d_K, A^{d_K}e^{-E}, n) = \frac{1}{750} e^{-E \cdot f(n)}\lp 7.6 e^{0.46} A^{f(n)} \Pi (n) \rp^{d_K}.
\end{align*}
\par For a certain choice of the pair $(A,E)$ from Odlyzko's table, which will be made more explicit later, we now make the following claims. 
\begin{lemma}\label{threeclaimsforthreeproofs}
There exists a choice of parameters $(A,E)$ such that the following claims hold:
\begin{itemize}
\item[(a)] For $n \geqslant 2$, the function $n \mapsto \Pi (n)^{-1} O(n,2,A,E)$ is increasing. 
\item[(b)] For $n \geqslant 3$, the function $d_K \mapsto O(n, d_K, A, E)$ is increasing.
\item[(c)] We have $\Pi(4)^{-1} O(4,2,A,E) > 1.83$.
\end{itemize}
\end{lemma}
Together, these claims imply that for any $n \geqslant 4$, the lattice $\Gamma$ must “come from $\Q$” in the sense discussed earlier. Indeed, we know that the covolume of $\Gamma_0$ in $\SP_\yoyo (\R)$ is smaller than $1.83 \cdot \Pi (n)$. If $n \geqslant 4$ and $K \neq \Q$, so that $\lb K : \Q \rb = d_K \geqslant 2$, the three claims of the lemma then imply that
\begin{align*}
    \Pi (n)^{-1} O(n, d_K, A, E) \geqslant \Pi (n)^{-1} O(n, 2, A, E) \geqslant \Pi (4)^{-1} O(4,2, A, E) > 1.83,
\end{align*}
and consequently $O(n, d_K, A, E) > 1.83 \cdot \Pi (n)$. In view of \eqref{referencecovolume}, \eqref{09012025-numericalupperbound}, and \eqref{odlyzkocovolbound}, this implies that for $n \geqslant 4$, a number field $K$ of degree greater than $1$ cannot give rise to a lattice $\Gamma$ of \textit{minimal} covolume. Hence, under these assumptions, $d_K = 1$ and $K = \Q$, as claimed. 
\begin{proof}[Proof of Lemma \ref{threeclaimsforthreeproofs}]
To see part \textit{(a)}, suppose $n \geqslant 1$ is arbitrary. Then we compute that
\begin{align*}
   &\log \lp \frac{\Pi(n+1)^{-1} O(n+1, 2, A, E)}{\Pi(n)^{-1} O(n, 2, A, E)} \rp \\ 
   &\quad \quad = -E \cdot (2n+3/2) + (4n+3) \log A + \log \frac{(2n+1)!}{(2 \pi)^{2n+2}} \\[+0.4em]
   &\quad \quad > -E \cdot (2n+3/2) + (4n+3) \log A - (2n+2) \log 2 \pi \\[+0.7em]
   &\quad \quad \quad \quad + \tfrac{1}{2} \log 2 \pi (2n+1) + (2n+1) \log (2n+1) - (2n+1),
\end{align*}
where we used the inequality \eqref{pointwisestirling}. Writing $\Omega(A,E) = 4 \log A - 2 E - 2 \log 2 \pi - 2$ and rearranging the terms on the right-hand side of the estimate above, we find that
\begin{align*}
&\log \lp \frac{\Pi(n+1)^{-1} O(n+1, 2, A, E)}{\Pi(n)^{-1} O(n, 2, A, E)} \rp 
\\[+0.5em] &\quad \quad > n \lp 2 \log (2n+1) + \Omega (A, E) \rp + \tfrac{3}{2} \log (2n+1) + 3 \log A - \tfrac{3}{2} E - \tfrac{3}{2} \log 2 \pi - 1 
\\[+0.5em] &\quad \quad > n ( 2 \log (2n + 1) + \Omega(A,E) ) + \tfrac{3}{2} \log (2n+1) + \tfrac{3}{4} \Omega(A,E).
\end{align*}
Since our claim will follow if the logarithm we are estimating is non-negative, we derive a condition on $\Omega(A,E)$ which expresses this. The right-hand side above is non-negative if and only if
\begin{align*}
    \Omega(A,E) \cdot (n + 3/4) + 2n \log (2n+1) + \tfrac{3}{2} \log (2n+1) \geqslant 0,
\end{align*}
which is equivalent to
\begin{align}\label{conditiononOmega}
    4 \log A - 2 E \geqslant 2 \log 2 \pi + 2 - 2 \log (2n+1).
\end{align}
Since the right-hand side of \eqref{conditiononOmega} is decreasing in $n$, we conclude that our claim is justified if we can choose $A$ and $E$ such that
\begin{align}\label{ClaimAequiv}
    2 \log A - E \geqslant \log 2 \pi + 1 - \log 5 \approx 1.2284.
\end{align}
We postpone the matter of choosing $A$ and $E$ such that \eqref{ClaimAequiv} is satisfied until a later point. \par 
To prove part \textit{(b)}, we need to show that there is a choice of $A > 1$ such that for $n \geqslant 3$, $7.6 e^{0.46} A^{f(n)} \Pi (n) \geqslant 1$, or equivalently,
\begin{align}\label{ClaimBequiv}
    \log A \geqslant \frac{-\log \Pi (n) - \log 7.6 - 0.46}{f(n)}.
\end{align}
By computing approximations of the first values of $\Pi(n)$ and using Stirling's formula, one can easily see that $\Pi (n)$ is increasing for $n \geqslant 14$. Since $f(n)$ is certainly also increasing for $n \geqslant 14$, we only need to choose $A$ such that \eqref{ClaimBequiv} holds for $n = 3, \ldots, 14$. By using lower bounds for the values of $\Pi (n)$ and inspecting the inequalities \eqref{ClaimBequiv} for these values of $n$, we find that $A$ satisfies the claim \textit{(b)} provided that
\begin{align}\label{ClaimBanswer}
    A > 5.66.
\end{align}
Provided that such an $A$ exists which also satisfies \eqref{ClaimAequiv}, the claim is proved. 
\par We now prove part \textit{(c)}. By using the approximation $\Pi(4) \approx 3.9 \cdot 10^{-10}$, we see that this claim is satisfied if 
\begin{align*}
    e^{-E \cdot f(4)} A^{2 \cdot f(4)} \Pi (4) > 1.83 \cdot 750 \cdot \frac{1}{7.6^2} \cdot e^{-0.92} \approx 9.4697.
\end{align*}
In order for this to hold, it is sufficient that 
\begin{align}\label{ClaimCequiv}
    -E + 2 \log A > \frac{\log 9.47 - \log \Pi (4)}{f(4)}.
\end{align}
The proof will be concluded once we show that we can choose $A$ and $E$ that satisfy this inequality while also satisfying \eqref{ClaimBanswer} and \eqref{ClaimAequiv}. \par 
We now search through the possible pairs $(A,E)$ to find a choice of $A$ and $E$ satisfying the three requirements \eqref{ClaimCequiv}, \eqref{ClaimBanswer}, and \eqref{ClaimAequiv}. It turns out that we may take 
\begin{align*}
    (A,E) = (6.894, 2.2667),
\end{align*}
and the lemma is proved.
\end{proof}
\subsubsection*{II. The case $\bm{n = 3}$}
When $n = 3$ we are no longer able to rule out all the cases $d_K \geqslant 2$ in a uniform way. Rather, we will have to consider different regimes of $d_K$ and argue accordingly. Hence, our first order of business is to determine these regimes. \par 
Initially, we ask how large $d_K$ has to be in order that
\begin{align*}
\Pi(n)^{-1} O(n,d_K,A,E) = \Pi(3)^{-1} O(3,d_K,A,E) > 1.83.
\end{align*}
Since $\Pi (3) = 45 / 256 \pi^{12}$, we have
\begin{align*}
\Pi(3)^{-1} O(3,d_K,A,E) &= \frac{1}{750} \Pi(3)^{-1}  e^{-E \cdot 7.5} \lp 7.6 e^{0.46} A^{7.5} \cdot \Pi (3) \rp^{d_K} \\[+1em]
&= \frac{256 \pi^{12}}{33750} \cdot  e^{-7.5E}\lp e^{0.46} A^{7.5} \frac{342}{256 \pi^{12}} \rp^{d_K}.
\end{align*}
This quantity exceeds $1.83$ precisely when
\begin{align*}
   d_K \lp 7.5 \log A + 0.46 + \log \frac{342}{256 \pi^{12}} \rp - 7.5E > \log \frac{1.83 \cdot 33750}{256 \pi^{12}},
\end{align*}
and for $A > 5.65$ (ensuring that the coefficient of $d_K$ is positive), this is certainly satisfied when
\begin{align}\label{ineqthreepointthirteen}
    d_K > \frac{ 7.5E - 8.25}{7.5 \log A - 12.99}
\end{align}
because of the approximations
\begin{align*}
0.46 + \log \frac{342}{256 \pi^{12}} \approx -12.987, \quad \quad \log  \frac{1.83 \cdot 33750}{256 \pi^{12}} \approx -8.251.
\end{align*}
\par In conclusion, \eqref{ineqthreepointthirteen} is a sufficient condition for the inequality $O(3,d_K, A,E) > 1.83 \cdot \Pi(3)$ to hold. The task now, therefore, is to choose the pair $(A,E)$ from Odlyzko's table in such a way as to \textit{minimize} the right-hand side of \eqref{ineqthreepointthirteen}. Using a computer algebra system, we find that the minimal value is $3.31$ and is attained at $(A,E) = (13.047, 3.8667)$. It follows that a number field of degree at least $4$ cannot give rise to the lattice $\Gamma$.\par 
To completely settle the case $n = 3$ we must now exclude the two remaining undesirable possibilities $d_K = 2$ and $d_K = 3$. To this end, assuming that $\Gamma$ has covolume smaller than $\mu \lp \SP_6(\R) / \Gamma_0 \rp = \zeta(2) \zeta(4) \zeta(6) \Pi (3) < 1.83 \cdot \Pi (3)$, we see from Corollary \ref{toyotacorollary} that
\begin{align*}
D_K < \lp 1372.5 \cdot \Pi (3)  \Big( 7.6 e^{0.46} \Pi(3) \Big)^{-d_K} \rp^{1/7.5} \approx
\begin{cases}
10.63 &\text{ if $d_K = 2$}, \\
60.09 &\text{ if $d_K = 3$}.
\end{cases}
\end{align*}
By consulting the $L$-functions and modular forms database \cite{lmfdb}, we find that $K$ is restricted to being one of the following number fields:
\begin{itemize}
    \item If $d_K = 2$, there are two possible number fields, both with class number $1$, with respective discriminants $D_K = 5, 8$. In light of the additional information that $h_K = 1$, the estimate \eqref{protoD} in Corollary \ref{protodbound} can be applied to these fields to yield an even sharper bound on $D_K$. Thus, with $(n,d_K) = (3,2)$, we find that in fact $D_K < 5.27$, so that $D_K = 5$ and $K = \Q \lb \sqrt{5} \rb$ is the only possibility. We postpone the matter of conclusively excluding this number field until a later point when we have determined the possibilities in the case $n = 2$ as well.
    \item If $d_K = 3$, the only possibility is $D_K = 49$, in which case one also has $h_K = 1$. Knowing the class number, we can again apply Corollary \ref{protodbound} and obtain the strengthened bound $D_K < 28.087$. This is impossible since any totally real cubic number field must have discriminant at least $49$.
\end{itemize}
\subsubsection*{III. The case $\bm{n = 2}$}
We are now forced to refine our covolume estimates in order to get the implied bounds on $d_K$ within a range where inspection in number field databases is a viable way forward. We thus begin by recalling from Lemma \ref{LLLemmaxyx} and \eqref{covolest02} that, when $n = 2$,
\begin{align}\label{bigineqnumberPI}
\nonumber \mu (\G(K_{\upsilon_0}) / \Gamma) &> \frac{D_K^{5 - (t+1)/2}}{25 t (t+1)} \lp \frac{3e^{0.46}}{64 \pi^6} \pi^{(t+1)/2} \Gamma \lp \frac{t+1}{2} \rp^{\! \! -1} \zeta(t+1)^{-1} \rp^{d_K} \\
&= \frac{e^{-5E+E(t+1)/2 }}{25 t (t+1)}  \lp  \frac{3e^{0.46}}{64 \pi^6} A^{4.5 - t/2} \pi^{(t+1)/2} \Gamma \lp \frac{t+1}{2} \rp^{\! \! -1} \zeta(t+1)^{-1} \rp^{d_K},
\end{align}
where we also used the generic bound \eqref{odlyzkobounds} with $(A,E)$ an unspecified pair from Odlyzko's table \cite{odlyzkowebsite}.
For specific choices of $A$, $E$, and $t$, we initially want to show that for $d_K$ larger than a certain threshold, the right-hand side of \eqref{bigineqnumberPI} exceeds the covolume $\zeta(2) \zeta(4) \Pi(2)$ of $\SP_4(\Z)$ in $\SP_4(\R)$. Moreover, we wish to choose $(A,E,t)$ such that this threshold is as low as possible. Getting such a lower bound on $d_K$ requires us to guarantee in some way that the base of the exponential expression of the degree above is strictly bigger than $1$, i.e.
\begin{align}\label{baseshouldbebiggerthanone}
    \frac{3e^{0.46}}{64 \pi^6} A^{4.5 - t/2} \pi^{(t+1)/2} \Gamma \lp \frac{t+1}{2} \rp^{\! \! -1} \zeta(t+1)^{-1} > 1.
\end{align}
Since the search for the optimal parameters $(A,E)$ will involve a computer algebra system in any case, we merely add the condition \eqref{baseshouldbebiggerthanone} into our search parameters and carry out the following program:\\
\begin{adjustwidth}{2.5em}{0pt}
\begin{lstlisting}
(Python 3.11)
# The bound 5.5535611217287 was obtained with A , E , t = 21.512 , 6.0001 , 1.2000000000000002

# List of values of A in Odlyzko's Table 4 of discriminant bounds
Alist = [ ... ]

# List of values of E in Odlyzko's Table 4 of discriminant bounds
Elist = [ ... ]

eta = 3*math.exp(0.46)/(64*(math.pi)**6)
standardcovolumetwo = scipy.special.zeta(2)*scipy.special.zeta(4)*(3/(32*(math.pi)**6))

def alpha(s):
    return (math.pi)**(s/2)*(scipy.special.gamma(s/2))**(-1)*(scipy.special.zeta(s))**(-1)

def rhs(A,E,t): 
    logXcoeff = E*(t+1)/2-5*E - math.log(25*t*(t+1))
    return (math.log(standardcovolumetwo) - logXcoeff)/math.log((eta*A**(4.5-t/2)*alpha(t+1)))

def optimizer():
    minvalue = 58
    for i in range(0,len(Alist)):
        A = Alist[i]
        E = Elist[i]
        for increment in range(1,250):
            t = 0.1*increment
            if eta*A**(4.5-t/2)*alpha(t+1) > 1: # ensuring the base of the exp. fct. is > 1
                if rhs(A,E,t) < minvalue:
                    minvalue = rhs(A,E,t)
                    abest = A
                    ebest = E
                    tbest = t
    print('The bound', minvalue, 'was obtained with A , E , t =', abest, ',', ebest, ',', tbest)
\end{lstlisting}
\end{adjustwidth}
\vspace{1em}
(By fine-tuning the search for an optimal value of $t$, given that this value is $\approx 1.2$, we made only insignificant improvements on the resulting bound on $d_K$.) Thus, with the choices
\begin{align}
    (A,E,t) = (21.512, 6.0001, 1.2),
\end{align}
we get that $\mu(\G_\infty / \Gamma) > \Psi(2)$ whenever $d_K \geqslant 6$. Accordingly, we know that the number field $K$ which realizes our lattice of minimal covolume must have degree $d_K \in \lbr 1, 2, 3, 4, 5 \rbr$. \par 
We now translate this bound on $d_K$ into bounds on the discriminant $D_K$, which will force $K$ to belong to a small list of possible number fields. To this end, we compute that $\Pi(2) = 3 / 32 \pi^6$ and insert the optimized parameter $t = 1.2$ into \eqref{covolest02} to obtain
\begin{align*}
\mu \lp \G (K_{\upsilon_0}) / \Gamma \rp 
> \frac{D_K^{3.9}}{66} \lp \frac{3 e^{0.46}}{64 \pi^6} \alpha(2.2) \rp^{d_K} 
> \frac{D_K^{3.9}}{66} \cdot 0.00019^{d_K},
\end{align*}
where we also computed that
\begin{align*}
\frac{3 e^{0.46}}{64 \pi^6} \alpha(2.2) \approx 0.0001919\ldots > 0.00019. 
\end{align*}
By demanding that $\mu \lp \G (K_{\upsilon_0}) / \Gamma \rp < \zeta(2) \zeta(4) \Pi(2) = 1/5760$, we obtain the following upper bounds on $D_K$:
\begin{align}\label{DKboundn2}
D_K < \lp \frac{11 \cdot 0.00019^{-d_K}}{960} \rp^{1/3.9} 
\approx \begin{cases}
25.74 &\text{ if } d_K = 2, \\
231.65 &\text{ if } d_K = 3, \\
2084.50 &\text{ if } d_K = 4, \\
18757.18 &\text{ if } d_K = 5.
\end{cases}
\end{align}
\par Consulting the $L$-functions and modular forms database \cite{lmfdb}, we can translate the discriminant bounds given in \eqref{DKboundn2} into the following list of concrete possibilities for $K$. (The information given is enough to uniquely identify $K$ in each case.)
\begin{itemize}
	\item If $d_K = 5$, the bound \eqref{DKboundn2} points to a unique totally real number field, which has discriminant $D_K = 14641$ and class number $h_K = 1$. Inserting this information into Corollary \ref{protodbound}, we get the refined bound $D_K \leqslant 3177$, which makes this case impossible.
    \item If $d_K = 4$, we find six possible number fields, all with class number $1$, with respective discriminants $D_K = 725, 1125, 1600, 1957, 2000, 2048$. Appealing again to Corollary \ref{protodbound}, we find that for a quartic extension with class number $1$, necessarily $D_K \leqslant 436$, which makes all five cases impossible. 
    \item If $d_K = 3$, we find five possible number fields, all with class number $1$, with respective discriminants $D_K = 49, 81, 148, 169, 229$. Corollary \ref{protodbound} implies that, with $h_K = 1$, the field $K$ must have discriminant $D_K \leqslant 59$. Hence only the case $D_K = 49$ is possible.
    \item If $d_K = 2$, we find seven possible number fields, all with class number $1$, with respective discriminants $D_K = 5, 8, 12, 13, 17, 21, 24$. By Corollary \ref{protodbound}, one even has $D_K \leqslant 8$ in this case; hence only the cases $D_K = 5, 8$ are possible.
\end{itemize}
\par Now, for each possible number field $K$, knowing its relevant arithmetic invariants allows us to compute the explicit lower bound $S(\Lambda)$ for the covolume of the principal arithmetic subgroup $\Lambda$. By Lemma \ref{covolprotobound} and \eqref{08012025-gammacovolume}, we can then obtain a corresponding lower bound for the covolume of $\Gamma$. By comparing this lower bound to the covolume 
\begin{align*}
\Psi(n) = \mu \lp \SP_\yoyo(\R)/\Gamma_0 \rp = \begin{cases}
1/5760\approx 1.736 \cdot 10^{-4} &\text{ if } n = 2, \\
1/2903040 \approx 3.445 \cdot 10^{-7} &\text{ if } n = 3,
\end{cases}
\end{align*}
we end up with an \textit{upper bound} on the possible contribution of the index $\lb \Gamma : \Lambda \rb$ and the parahoric factors $e' (P_\upsilon)$ that were not taken into account so far. Since all of the possible fields that have not yet been ruled out have class number $1$, we use Lemma \ref{covolprotobound} with $h_K = 1$ and determine the values of $2^{1-2d_K} S(\Lambda)$ for all the possible combinations of degrees and discriminants. Additionally, as the inequality $2^{1-2d_K} S(\Lambda) \leqslant \Psi(n)$ is necessary in order for $\Gamma$ to have minimal covolume in $\SP_\yoyo (\R)$, we also compute the quotient $\Psi (n) / \lp 2^{1-2d_K} S(\Lambda) \rp$ in each case.  \par 
When $n = 2$, we have $\Pi (n) = \Pi(2) = 3/32 \pi^6$, and 
\begin{align*}
\frac{S(\Lambda)}{2^{2 d_K - 1}} = 2 D_K^{5} \lp \frac{3}{128 \pi^6} \rp^{d_K} \zeta_K (2) \zeta_K (4),
\end{align*}
which we evaluate as follows:
\begin{itemize}
\item If $d_K = 3$ and $D_K = 49$, we find that
\begin{align*}
\frac{S(\Lambda)}{2^{2 d_K - 1}} \approx 8.75 \cdot 10^{-6},
\end{align*}
and the quotient of $\Psi (n) = \Psi(2) = 1/5760$ and this approximate value of $2^{1-2d_K} S(\Lambda)$ is approximately $19.85$.
\item If $d_K = 2$ and $D_K = 5$, we have
\begin{align*}
\frac{S(\Lambda)}{2^{2 d_K - 1}} \approx 4.34 \cdot 10^{-6},
\end{align*}
and the quotient of $1/5760$ and this approximate value of $2^{1-2d_K} S(\Lambda)$ appears to be exactly equal to $40$.
\item If $d_K = 2$ and $D_K = 8$, we have
\begin{align*}
\frac{S(\Lambda)}{2^{2 d_K - 1}} \approx 5.97 \cdot 10^{-5}, 
\end{align*}
and the quotient of $1/5760$ and this approximate value of $2^{1-2d_K} S(\Lambda)$ is approximately equal to $2.91$.
\end{itemize}
When $n = 3$, on the other hand, 
\begin{align*}
\frac{S(\Lambda)}{2^{2 d_K - 1}} = 2 D_K^{10.5} \lp \frac{45}{1024 \pi^{12}} \rp^{d_K} \zeta_K (2) \zeta_K (4) \zeta_K (6),
\end{align*}
and we compute that
\begin{itemize}
\item if $d_K = 2$ and $D_K = 5$, then
\begin{align*}
\frac{S(\Lambda)}{2^{2 d_K - 1}} \approx 1.15 \cdot 10^{-7}. 
\end{align*}
The quotient of $1/2903040$ and this approximate value of $2^{1-2d_K} S(\Lambda)$ is approximately equal to $2.99$.
\end{itemize}
\par Since all the quotients computed above exceed $1$, we are forced to make one final refinement to rule out all but one of these number fields. Namely, for the fields listed above, it is not too difficult to compute the index related to the groups $U_K^+$ and $U_K^2$ which appear in the first inequality of Lemma \ref{covolprotobound}. Since the quotient of the middle part and the right-hand side of the lemma is $c(K) = 2^{2 d_K - 1} / \lb U_K^+ : U_K^2 \rb$, we adjust the quotients computed above accordingly by multiplying the values obtained by $1/c(K)$. For $n = 2$, we can argue as follows: 
\begin{itemize}
\setlength{\parindent}{2em}
\item If $d_K = 3$ and $D_K = 49$, we computed the quotient $19.85$. Since $c(K) = 2^5 / \lb U_K^+ : U_K^2 \rb$, the adjusted quotient is $\lb U_K^+ : U_K^2 \rb \cdot 19.85 / 32$, which is less than $1$ (thus making this case impossible) if and only if $\lb U_K^+ : U_K^2 \rb = 1$. We now argue that this is indeed the case: \\[-1.6em]\par $K = \Q ( \alpha)$ where $\alpha = 2 \cos (\tfrac{2 \pi}{7}) \approx 1.25$ has minimal polynomial $X^3 + X^2 - 2X - 1$ and conjugates $\alpha' = -1/(\alpha + 1)$ and $\alpha '' = -1/(\alpha'+1) = -(1+1/\alpha)$. In this case, one may check (e.g. with the help of \cite{lmfdb}) that two fundamental units are $\varepsilon_1 = \alpha$ and $\varepsilon_2 = \alpha^2 - 1$. We now show that the group $U_K^+$ of totally positive units is generated by $\varepsilon_1^2$ and $\varepsilon_2^2$, which will prove that the index of $U_K^2$ in $U_K^+$ is $1$.\\[-1.6em]\par Suppose that $u = \pm \varepsilon_1^{\ell_1} \varepsilon_2^{\ell_2} \in \ringint_K^\times$ is totally positive. Since $\varepsilon_1, \, \varepsilon_2 > 0$, we must have $u = \varepsilon_1^{\ell_1} \varepsilon_2^{\ell_2}$. Moreover, by dividing by the totally positive number $\varepsilon_1^{2 \lfloor \ell_1 / 2 \rfloor} \varepsilon_2^{2 \lfloor \ell_2 / 2 \rfloor}$, we can replace $\ell_1$ and $\ell_2$ with their remainders (mod $2$); that is, we assume that $\ell_1 , \ell_2 \in \lbr 0, 1 \rbr$. If $\ell_1 = \ell_2 = 1$, in which case $u = \alpha (\alpha^2 - 1)$,  $u$ is not totally positive, as its conjugate
\begin{align*}
\alpha'' \lp (\alpha'')^2 - 1 \rp = -(1+1/\alpha)\lp 2/\alpha + 1/\alpha^2 \rp
\end{align*}
is negative. Therefore either $\ell_1 = 0$ or $\ell_2 = 0$, corresponding to the two possibilities $u=\varepsilon_2$ or $u = \varepsilon_1$. However, neither of these cases are possible: The conjugate $\alpha'$ of $\varepsilon_1 = \alpha$ is negative; and for $\varepsilon_2 = \alpha^2 - 1$, one sees that its conjugate
\begin{align*}
(\alpha')^2 - 1 = 1/(\alpha + 1)^2 - 1 < 1 - 1 < 0
\end{align*}
is also negative. In conclusion, $\ell_1 = \ell_2 = 0$, as claimed.
\item If $d_K = 2$ and $D_K = 5$, we computed the quotient $40$. In this case we can check that $\lb U_K^+ : U_K^2 \rb = 1$ so that $c(K) = 2^3 / \lb U_K^+ : U_K^2 \rb = 8$, and hence the adjusted quotient is $5$. Indeed, this follows from \eqref{quadraticfieldnottotallypositiveFU} since, by \eqref{quadnffundamentalunit}, the fundamental unit in this case is $\varepsilon = (1+\sqrt{5})/2$ (corresponding to $a^2 - 5 b^2 = -4$ with $(a,b) = (1,1)$), which is not totally positive, as the non-trivial embedding $K \hookrightarrow \R$ maps $\varepsilon$ to $(1-\sqrt{5})/2 < 0$. In summary, the current case $(d_K, D_K) = (2,5)$ has yet to be ruled out.
\item If $d_K = 2$ and $D_K = 8$, we computed the quotient $2.91$. Once again, we can check that $\lb U_K^+ : U_K^2 \rb = 1$, which yields an adjusted value of $\lb U_K^+ : U_K^2 \rb \cdot 2.91 / 8 < 1$ and rules out this case. Indeed, for this field we have $\varepsilon = (2 + 2 \sqrt{2})/2 = 1 + \sqrt{2}$ (corresponding to $a^2 - 8 b^2 = -4$ with $(a,b) = (2,1)$), which is not totally positive, and we can then argue as in the previous case.
\end{itemize}
Finally, for $n = 3$, we can argue as follows:
\begin{itemize}
\item If $d_K = 2$ and $D_K = 5$, we computed the quotient $2.99$. As we observed in the second case above, the index $\lb U_K^+ : U_K^2 \rb$ is $1$, and hence the adjusted quotient becomes $2.99 / 8 < 1$. This case has therefore been ruled out.
\end{itemize}
\section{Local Considerations and the Final Steps}\label{sectionlocalcontributions}
We have now singled out the number field $K = \Q (\sqrt{5})$ as the only candidate other than $\Q$ that can give rise to the lattice $\Gamma$ (and only for $n = 2$). As we cannot shed any more light on this matter with the global methods we have used up until this point, we will now take the parahoric factors coming from the finite places of $K$ into account. This will allow us to conclusively rule out the field $\Q (\sqrt{5})$ and prove Theorem 1.1 after a detailed analysis of the only remaining case $K = \Q$.
\subsection{Estimates of the parahoric factors} \label{subsection-takingtheparfacts}
For $K = \Q (\sqrt{5})$ and $n = 2$, Theorem \ref{prasadinourspecialcase} shows that
\begin{align*}
\mu (\G (K_{\upsilon_0}) / \Lambda ) = \frac{28125}{1024 \pi^{12}} \, \zeta_K (2) \zeta_K (4) \prod_{\upsilon < \infty} e'(P_\upsilon),
\end{align*}
since $D_K = 5$ and $\Pi (2) = 3/32 \pi^6$. In terms of the covolume of $\Gamma$, this says that
\begin{align*}
\mu (\G (K_{\upsilon_0}) / \Gamma ) = \lb \Gamma : \Lambda \rb^{-1} \frac{28125}{1024 \pi^{12}} \, \zeta_K (2) \zeta_K (4) \prod_{\upsilon < \infty} e'(P_\upsilon).
\end{align*}
From our discussion in the previous section, it follows that $\Gamma$ has minimal covolume if and only if the combined contribution of the index $\lb \Gamma : \Lambda \rb$ and the parahoric factors does not exceed $5$. That is, we can rule out the field $K$ if we can show that
\begin{align}\label{biggerthanfiveisnice}
\lb \Gamma : \Lambda \rb^{-1} \prod_{\upsilon < \infty} e'(P_\upsilon) > 5.
\end{align}
The stronger bound in \eqref{indexboundgaloiscohomology} now states that $\lb \Gamma : \Lambda \rb \leqslant 2^{\# \mathcal{T}}$, so \eqref{biggerthanfiveisnice} will certainly follow if we can show that
\begin{align}\label{biggerthanfiveNO2}
\prod_{\upsilon < \infty} e'(P_\upsilon) > 5 \cdot 2^{\# \mathcal{T}},
\end{align}
where $\mathcal{T}$ denotes the set of finite places $\upsilon$ where $P_\upsilon$ is not hyperspecial. \par 
It is possible, using Bruhat--Tits theory, to describe the Levi components $\textbf{M}_\upsilon$ for each parahoric subgroup $P_\upsilon \subseteq \G (K_\upsilon)$ explicitly, and hence to compute the exact values of $e'(P_\upsilon)$ with the help of \eqref{estrich}. Conveniently, for the case of a $K$-form of $\SP_4$, these computations already exist in the literature. Namely, we obtain from \cite[Sect. 3, Table 2]{ganhankeyu} that for any place $\upsilon < \infty$ of $K$, one either has $\textbf{M}_\upsilon = \SP_2 (\F_{q_\upsilon}) \times \, {}^2 O_2 (\F_{q_\upsilon}) = \SL_2 (\F_{q_\upsilon}) \times \, {}^2 O_2 (\F_{q_\upsilon})$ or $\textbf{M}_\upsilon = \SP_4 (\F_{q_\upsilon})$. (Here ${}^2 O_2 $ denotes the non-split, quasisplit orthogonal group.) By computing the orders of these groups and using \eqref{estrich} (cf. also \cite[eq. (2.12)]{ganhankeyu}, one then has the two possibilities
\begin{align}\label{two_alternatives_from_heaven}
e'(P_\upsilon) = 1 \quad \text{or} \quad e'(P_\upsilon) = T(q_\upsilon) := \frac{q_\upsilon^4-1}{2(q_\upsilon + 1)}.
\end{align}
By the discussion at the end of Section \ref{subsubsectionparahoricsubgroups}, we note that essentially \eqref{biggerthanfiveNO2} can only fail if the number of parahorics which are not hyperspecial is very small. Indeed, every $P_\upsilon$ with $e'(P_\upsilon) \neq 1$ contributes to the right-hand side of \eqref{biggerthanfiveNO2} with a factor $2$, whereas the left-hand side receives a contribution of a factor $\approx q_\upsilon^3$, which exceeds $2$ already for small values of $q_\upsilon$. \par 
More precisely, we check that $T(2) = 2.5$, $T(3) = 10$, and $T(x) > 25$ for $x \geqslant 4$. Therefore, if \eqref{biggerthanfiveNO2} fails, then $P_\upsilon$ is hyperspecial for any finite place $\upsilon$ with $q_\upsilon \geqslant 4$. Hence, the only possibilities for non-hyperspecial $P_\upsilon$ come from residue fields of order $q_\upsilon = 2$ or $q_\upsilon = 3$. However, neither of these possibilities can occur for $K = \Q (\sqrt{5})$, as we will now demonstrate. \par 
If $q_\upsilon = 2$, then $K_\upsilon = \Q_2 \lb X \rb / \langle X^2 - 5 \rangle$ is a quadratic extension of $\Q_2$ since $5$ is not a square in $\Q_2$. (It is well-known that $x = 2^n u$ with $u \in \Z_2^\times$ is a square in $\Q_2$ if and only if $n$ is even and $u \equiv 1$ (mod $8$), which is not the case for $x = 5$.) Consequently, the local degree of $K$ at a place $\upsilon$ lying above $2$ is $q_\upsilon = 4$. \par 
On the other hand, if $q_\upsilon = 3$, then $K_\upsilon$ must also be a quadratic extension, as $x = p^n u$ (with $u \in \Z_p^\times$) is a square in $\Q_p$ if and only if $n$ is even and $u$ (mod $p$) is a quadratic residue. Since $5$ (mod $3$) is a quadratic non-residue, our claim follows, and the local degree is $q_\upsilon = 9$.\par 
We have now established that $e'(P_\upsilon) = 1$ for all finite places $\upsilon$ of $K$, meaning that $\G$ splits at all finite places by \eqref{09012025-convenient}. To conclusively rule out $K$, we note that since $2n = 4$, our discussion in Section \ref{subsectionalgebraicgroups} shows that $\G$ must be a special unitary group over quaternion algebra $\mathbb{H}_K$. Moreover, if $\upsilon$ denotes any place of $K$, we recall from \cite[Lemma 2.3]{emerykim} that $\SU(\mathbb{H}_{K_\upsilon}; \tau; h)$ is ${K_\upsilon}$-split if and only if the quaternion algebra $\mathbb{H}_{K_\upsilon}$ splits; that is, if and only if there is an isomorphism defined over ${K_\upsilon}$ such that
\begin{align*}
\mathbb{H}_{{K_\upsilon}} \simeq \mathrm{Mat}_{2 \times 2} ({K_\upsilon}).
\end{align*}
\par By \cite[Thm. 14.6.1]{voight}, we then conclude that, because $\G$ is split at all finite places (which means that $\mathbb{H}_K$ splits at all finite places), the number of infinite places where $\mathbb{H}_K$ is ramified (i.e non-split) is either $0$ or $2$. It cannot be the case that $\mathbb{H}_K$ is ramified at both real places of $K$, since $\G$ splits at $\upsilon = \upsilon_0$. Hence $\G$ must even be split at the other real place $\upsilon_1$ of $K$. However, this is impossible, as we have already established, since then $\G(K_{\upsilon_1}) = \SP_4(\R)$ is not compact.
\subsection{Identifying the lattice $\Gamma$}
In this final section we will show that $\G$ is isomorphic to the split form $\SP_\yoyo$ over $\Q$ and complete the proof of Theorem 1.1. \par 
By the arguments up until this point, we know that for $n \geqslant 2$, a lattice $\Gamma \subseteq \SP_\yoyo (\R)$ of minimal covolume comes from $\Q$; that is, $\G$ is defined over $\Q$, and $\Gamma$ is the normalizer in $\G(\R)$ of the principal arithmetic subgroup $\Lambda = \G ( \Q) \cap \prod_{\upsilon < \infty} P_\upsilon$. 
\par Recall from Theorem \ref{prasadinourspecialcase} and \eqref{referencecovolume} the relation
\begin{align*}
\mu (\G_\infty / \Gamma ) = \lb \Gamma : \Lambda \rb^{-1} \mu (\G_\infty / \Gamma_0 ) \prod_{\upsilon < \infty} e' (P_\upsilon),
\end{align*}
and that for almost all places $\upsilon < \infty$, the parahoric $P_\upsilon$ is hyperspecial and $e'(P_\upsilon) = 1$. Under these circumstances, any contribution to the covolume of $\Gamma$ in $\G_\infty$ must come from $\lb \Gamma : \Lambda \rb$ and the local factors corresponding to non-hyperspecial parahoric subgroups, which are indexed by the set $\mathcal{T}$ (cf. the proof of Lemma \ref{covolprotobound}). \par 
In the case of $K = \Q$, Lemma \ref{covolprotobound} states that $\lb \Gamma : \Lambda \rb \leqslant 2^{\# \mathcal{T}}$. For the proof that $\G$ is $\Q$-split it will therefore be sufficient to show that any local factor $e'(P_\upsilon)$ not equal to $1$ is strictly larger than $2$. To this end, let $e'(P_\upsilon) > 1$ be any such factor. By \eqref{09012025-convenient} we then see that $\G$ does not split over $\Q_\upsilon$. Due to this, and since $P_\upsilon$ is special, we can use the explicit formulas for $e'(P_\upsilon)$ in \cite{emerykim}, which we record in the following lemma.
\begin{lemma}[{\cite[Lemma 4.1]{emerykim}}] Let $\upsilon$ be a finite place of $\Q$ and $P_\upsilon \subseteq \G(\Q_\upsilon)$ a special, non-hyperspecial parahoric subgroup. If $q_\upsilon$ denotes the size of the residue field of $\Q_\upsilon$, then
\begin{align*}
e'(P_\upsilon) = 
\begin{cases}
\prod_{j = 1}^n \lp q_\upsilon^j + (-1)^j \rp &\text{ if $n$ is odd},\\[+0.5em]
\prod_{j = 1}^m \lp q_\upsilon^{4j-2} -1 \rp &\text{ if $n = 2m$ is even}.
\end{cases}
\end{align*}
\end{lemma}
\begin{proof}
We only need to justify the formula in the case of even rank. However, this follows immediately from \cite[eq. (4.9)]{emerykim} once we split the product in the numerator into two products depending on the parity of the indexing variable $j$.
\end{proof}
With this lemma we can easily deduce that $e'(P_\upsilon) > 2$ for any place $\upsilon \in \mathcal{T}$. Indeed, we certainly have $q_\upsilon \geqslant 2$ for any $\upsilon$, so if $n \geqslant 2$ is odd, then $e'(P_\upsilon) \geqslant 1 \cdot 5 \cdot 7 = 35$. On the other hand, if $n$ is even, then $e'(P_\upsilon) \geqslant 3$. It follows that, if $\Gamma$ has minimal covolume, then $\G$ must split at all finite places. Since $\G$ also splits at the unique infinite place $\upsilon_0$, we conclude (as in the previous section) that $\G = \SP_\yoyo$. \\ \par 

To complete the proof of Theorem 1.1, all that remains is to show that $\Lambda = \Gamma = \Gamma_0$, where the last equality holds up to conjugation. \par 
For $\upsilon$ an arbitrary finite place of $\Q$, the fact that the parahoric subgroup $P_\upsilon$ is hyperspecial means that can find an element $g_\upsilon$ in 
\begin{align*}
\GSP_\yoyo (\Q_\upsilon) = \Big\{ g \in \GL_\yoyo (\Q_\upsilon) : g^\intercal \lp \begin{smallmatrix}
0 & I \\
-I & 0
\end{smallmatrix} \rp g = \chi (g) \lp \begin{smallmatrix}
0 & I \\
-I & 0
\end{smallmatrix} \rp \text{ for some } \chi(g) \in \Q_\upsilon^\times \Big\},
\end{align*}
such that $g_\upsilon P_\upsilon g_\upsilon^{-1} = \SP_\yoyo (\Z_\upsilon)$. Indeed, $\GSP_\yoyo$ acts transitively on the hyperspecial parahoric subgroups of the symplectic group (cf. \cite[Sect. 2.5]{titscorvallis}). Because of the topology on the restricted product $\mathbb{A}_\Q$ and the coherence of the collection $\lbr P_\upsilon : \upsilon < \infty \rbr$, we can assume that $g_\upsilon = I$ for all places $\upsilon$ except finitely many. Doing so, we thus obtain an element 
\begin{align*}
    g = \Bigl( 1, (g_\upsilon)_{\upsilon < \infty} \Bigr) \in \GSP_{2n} (\mathbb{A}_\Q).
\end{align*}
\par The class number of a split reductive group is at most the class number of any maximal split torus (cf. \cite[Corollary to Thm. 8.11]{platrap}). Since $\Q$ has class number $1$, it follows that $\GSP_{2n}$ has class number $1$ over $\Q$, and we have
\begin{align*}
\GSP_{2n} (\A_\Q) = \lp \GSP_{2n} (\R) \times \prod_{\upsilon < \infty} \GSP_{2n} (\Z_\upsilon) \rp \cdot \GSP_{2n} (\Q).
\end{align*}
In terms of the coordinate corresponding to a finite place $\upsilon$, this means that we can write $g_\upsilon = g_\upsilon' h$ where $g_\upsilon' \in \GSP_{2n} (\Z_\upsilon)$ and $h \in \GSP_{2n} (\Q)$. Since $h = (g_\upsilon ')^{-1} g_\upsilon$ and conjugation by $g_\upsilon '$ leaves $\SP_\yoyo (\Z_\upsilon)$ invariant, so that $h P_\upsilon h^{-1}$ simply equals $\SP_\yoyo (\Z_\upsilon)$, we now finally obtain
\begin{align*}
    h \Lambda h^{-1} = h \SP_{2n}(\Q) h^{-1} \cap \prod_{\upsilon < \infty} h P_\upsilon h^{-1} = \SP_{2n}(\Q) \cap \prod_{\upsilon < \infty} \SP_{2n} (\Z_\upsilon) = \Gamma_0.
\end{align*}
\par We have established that $\Lambda = h^{-1} \Gamma_0 h$ for $h \in \mathrm{GSp}_\yoyo (\Q)$. Since $\Gamma$ is the normalizer of $\Lambda$ in $\SP_{2n} (\R)$, and $\Gamma_0 =  N_{\SP_\yoyo (\R)} (\Gamma_0)$ is its own normalizer (cf. \cite{allan}), we now conclude that
\begin{align*}
\Gamma = \Big\{ 
g \in \SP_\yoyo (\R) : h g h^{-1} \in \Gamma_0 \Big\} = \SP_\yoyo (\R) \cap h^{-1} \Gamma_0 h.
\end{align*}
In other words, $h \Gamma h^{-1} = \SP_\yoyo (\R) \cap \Gamma_0 = \Gamma_0$ where $h \in \mathrm{GSp}_\yoyo (\Q)$. Taking into account the isomorphism (now, \textit{automorphism}) $\SP_\yoyo (\R) \xrightarrow{\sim} \G(\R) = \SP_\yoyo (\R) $ induced by the distinguished real place $\upsilon_0$ of $K = \Q$, we conclude that $\Gamma$ is conjugate to $\Gamma_0$ by an element of $\mathrm{GSp}_\yoyo (\R)$ (cf. \cite{hua}). This concludes the proof of Theorem 1.1.
{\subsection{A final remark on non-special parahorics}
We have proved Theorem 1.1 under the assumption that every parahoric had maximal volume. We will now justify this assumption, as promised in Section \ref{subsubsectionparahoricsubgroups}. \par 
For the remainder of the paper, we will deviate from our previous notation and use the symbols $\Gamma$ and $\Lambda$ as free variables that denote, respectively, \textit{any maximal lattice} in $\SP(2n, \R)$ and the principal arithmetic subgroup it normalizes (by the maximality criterion \cite[Satz 3.5]{rohlfs} of Rohlfs). Likewise, $\G$ will denote the algebraic $K$-group that defines $\Lambda$. $P_\upsilon (\Lambda) = P_\upsilon (\Gamma)$ will denote the parahoric subgroup at $\upsilon$ associated to the pair $(\Lambda, \Gamma)$. Finally, we will also write $\Lambda^\mathrm{max}$ to denote a principal arithmetic subgroup for which every $P_\upsilon (\Lambda)$ is of maximal volume; and $\Gamma^\mathrm{max}$ will then denote the normalizer of $\Lambda^\mathrm{max}$. In the same vein, for a pair $(\Lambda, \Gamma)$, we will write $\Lambda^\mathrm{max}$ and $\Gamma^\mathrm{max}$ to denote the corresponding “maximized” principal arithmetic subgroup (as explained in \cite[Sect. 4.3]{belolipetskyemery} and \cite[Sect. 3.8]{borelprasad}) and its normalizer in $\G(\R)$, respectively.\par 
Up until this point, we have proved the following result in particular.
\begin{theorem}\label{theorem_when_parahorics_are_maximal}
Let $V_0$ denote the minimal covolume of all lattices in $\SP(2n, \R)$. If a lattice of the form $\Gamma^\mathrm{max}$ has covolume $V_0$, then $\G$ splits at every finite place of $K$, and every parahoric $P_\upsilon$ associated to $\Gamma^\mathrm{max}$ satisfies $e'(P_\upsilon) = 1$.
\end{theorem}
We now prove that any lattice of minimal covolume must, in fact, be of the form $\Gamma^\mathrm{max}$. 
\begin{lemma}\label{newlemma_covolume_must_decrease}
Suppose that $\Gamma \subseteq \SP(2n, \R)$ is a maximal lattice, associated to an algebraic $K$-group $\G$, with covolume $V_0$. Then $\Gamma = \Gamma^\mathrm{max}$; that is, every parahoric $P_\upsilon (\Gamma)$ has maximal volume among all parahoric subgroups of $\G(K_\upsilon)$.
\end{lemma}
\begin{proof}
The discussion in Section \ref{subsubsectionparahoricsubgroups} shows that also $\Gamma^\mathrm{max}$ has covolume $V_0$. Since both $\Gamma$ and $\Gamma^\mathrm{max}$ come from $\G$, Theorem \ref{theorem_when_parahorics_are_maximal} implies that $\G$ splits at all finite places --- in particular, a parahoric subgroup of $\G(K_\upsilon)$ is special if and only if it is hyperspecial, and consequently $\Gamma = \Gamma^\mathrm{max}$ if and only if every $P_\upsilon (\Gamma)$ is special. Theorem \ref{theorem_when_parahorics_are_maximal} also gives $e'(P_\upsilon(\Gamma^\mathrm{max})) = 1$, so we obtain from Theorem \ref{prasadinourspecialcase} and the inequality \cite[4.3.(15)]{belolipetskyemery} that
\begin{align}\label{gaarden_katholt}
1 = \frac{\mu (\G(\R) / \Gamma)}{\mu (\G(\R) / \Gamma^\mathrm{max})} = 
\frac{\lb \Gamma^\mathrm{max} : \Lambda^\mathrm{max} \rb}{\lb \Gamma : \Lambda \rb} 
\prod_{\upsilon < \infty} e'(P_\upsilon(\Gamma)) \geqslant \prod_{\upsilon < \infty} \frac{e' (P_\upsilon(\Gamma))}{\# \Xi_{\Theta_\upsilon}},
\end{align}
where $\Xi_{\Theta_\upsilon} \subseteq \mathrm{Aut}(\Delta_\upsilon)$ is the subgroup of all diagram automorphisms that preserve the type $\Theta_\upsilon = \Theta_\upsilon(P_\upsilon(\Gamma))$ of $P_\upsilon (\Gamma)$, and $\Delta_\upsilon$ denotes the local Dynkin diagram of $\G$ over $K_\upsilon$. \par Concretely, we have $\# \Xi_{\Theta_\upsilon} \in \lbr 1 , 2 \rbr$ for every $\upsilon$, and in fact $\# \Xi_{\Theta_\upsilon} = 1$ when $P_\upsilon (\Gamma)$ is special (see \cite[Sect. 3.1]{emerykim} or \cite[Sect. 3.2]{borelprasad}). Therefore, any factor corresponding to a special parahoric on the right-hand side of \eqref{gaarden_katholt} must be equal to $1$. The crucial fact, as we will see below, is that the converse holds as well: If $P_\upsilon (\Gamma)$ is not special, then the corresponding factor exceeds $1$. As this is impossible, the proof will be concluded once we demonstrate this claim. Equivalently, we must show that $e'(P_\upsilon (\Gamma)) > 2$ whenever $P_\upsilon (\Gamma)$ is not special. 
\par Let $\upsilon$ be any \textit{fixed} place such that $P_\upsilon (\Gamma)$ is not special. If the residue field of $K_\upsilon$ has degree $q_\upsilon = 2$ and $\G$ has rank $n = 2$, then our claim follows immediately from the inequality
\begin{align*}
e'(P_\upsilon (\Gamma)) \geqslant \frac{q_\upsilon^{4} - 1}{2(q_\upsilon + 1)} = \frac{15}{6} > 2,
\end{align*}
which follows by inspection of \cite[Sect. 3, Table 2]{ganhankeyu}. We therefore have either $n \geqslant 3$ or $q_\upsilon \geqslant 3$. In this situation, we can use the
“volume rigidity estimate” proved for non-special parahorics in \cite[Prop. 2.10.(iv)]{prasad} together with \eqref{campusfestival} (which is independent of the assumption we are in the process of justifying) to conclude that
\begin{align*}
e'(P_\upsilon (\Gamma)) \geqslant \frac{q_\upsilon^{n+1}}{q_\upsilon + 1} \prod_{j = 1}^n \lp 1 - \frac{1}{q_\upsilon^{2j}} \rp =: h(q_\upsilon, n).
\end{align*}
It is straightforward that $h(q_\upsilon , n)$ is increasing in both $q_\upsilon \geqslant 2$ and $n \geqslant 1$. In consequence, as $(q_\upsilon , n) \neq (2,2)$, we either have $h(q_\upsilon, n) \geqslant h(2,3) \approx 3.69$ or $h(q_\upsilon , n) \geqslant h(3,2) \approx 17.75$. In any case, we arrived at the desired conclusion.
\end{proof} }

\begin{small}
{\textbf{Acknowledgements.} The authors gratefully acknowledge (partial) funding from Deutsche For- schungsgemeinschaft via the grants DZ 105/02 and KO 4323/15. Moreover, the authors are grateful to F. Thilmany and A. Rapinchuk for helpful discussions, and to the two anonymous referees and the editor for several batches of valuable comments and suggestions that have considerably improved this paper.
}
\end{small}
\begin{footnotesize}

$\hrulefill$ \\\par 
\textsc{Department of Mathematics, Christian-Albrechts-Universität zu Kiel},\\ \textsc{24118-Kiel, Germany}. \textit{email address:}  [last name without accents]{\tt@math.uni-kiel.de} \\ \par 
\textsc{Department of Mathematics, Christian-Albrechts-Universität zu Kiel},\\ \textsc{24118-Kiel, Germany}. \textit{email address:} [last name]{\tt@math.uni-kiel.de} \\ \par 
\textsc{Department of Mathematics, Christian-Albrechts-Universität zu Kiel},\\ \textsc{24118-Kiel, Germany}. \textit{email address:} [last name with ö $\rightsquigarrow$ oe]{\tt@math.uni-kiel.de} \\ \par 
\end{footnotesize}
\end{document}